\documentclass[10pt,letterpaper,twoside]{article}
\usepackage{arxiv-preprint}

\RequirePackage{amsmath,amssymb,amsthm,mathtools,bm}
\RequirePackage{enumitem}
\usepackage[final]{showkeys}
\RequirePackage{url}
\RequirePackage{xcolor}
\definecolor{linkgreen}{RGB}{0,105,70}
\RequirePackage{hyperref}
\hypersetup{
	colorlinks=true,
	linkcolor=blue,
	citecolor=blue,
	urlcolor=blue,
	pdftitle={Berry--Esseen bounds for stochastic approximation with Markovian noise},
	pdfauthor={Zhijun Cai, Yujie Qin, Zhuosong Zhang},
	pdfsubject={Mathematics preprint},
	pdfkeywords={Berry--Esseen bound, stochastic approximation, Polyak--Ruppert averaging, Markov chain, Stein's method}
}
\RequirePackage[capitalise,nameinlink,noabbrev]{cleveref}

\startlocaldefs
\theoremstyle{plain}
\newtheorem{theorem}{Theorem}[section]
\newtheorem{lemma}[theorem]{Lemma}
\newtheorem{corollary}[theorem]{Corollary}

\theoremstyle{definition}

\newtheorem{assumption}{Assumption}[section]

\newtheorem{remark}{Remark}[section]
\endlocaldefs

\numberwithin{equation}{section}
\setlist[enumerate]{label=(\roman*),leftmargin=2.2em,itemsep=1mm,topsep=1mm}

\newcommand{\E}{\mathbb E}
\newcommand{\Pp}{\mathbb P}

\newcommand{\one}{\mathbf 1}

\newcommand{\cov}{\operatorname{Cov}}

\newcommand{\norm}[1]{\lVert #1\rVert}
\newcommand{\abs}[1]{\lvert #1\rvert}

\newcommand{\Pnu}{\Pp_{\nu}}

\crefname{section}{Section}{Sections}
\Crefname{section}{Section}{Sections}
\crefname{appendix}{Appendix}{Appendices}
\Crefname{appendix}{Appendix}{Appendices}
\crefname{subsection}{Section}{Sections}
\Crefname{subsection}{Section}{Sections}
\crefname{subsubsection}{Section}{Sections}
\Crefname{subsubsection}{Section}{Sections}
\crefname{equation}{equation}{equations}
\Crefname{equation}{Equation}{Equations}
\crefname{theorem}{Theorem}{Theorems}
\Crefname{theorem}{Theorem}{Theorems}
\crefname{lemma}{Lemma}{Lemmas}
\Crefname{lemma}{Lemma}{Lemmas}
\crefname{proposition}{Proposition}{Propositions}
\Crefname{proposition}{Proposition}{Propositions}
\crefname{corollary}{Corollary}{Corollaries}
\Crefname{corollary}{Corollary}{Corollaries}
\crefname{assumption}{Assumption}{Assumptions}
\Crefname{assumption}{Assumption}{Assumptions}
\crefname{example}{Example}{Examples}
\Crefname{example}{Example}{Examples}
\crefname{remark}{Remark}{Remarks}
\Crefname{remark}{Remark}{Remarks}
\crefname{enumi}{item}{items}
\Crefname{enumi}{Item}{Items}

\crefname{equation}{}{}
\Crefname{equation}{}{}

\crefformat{equation}{#2\textup{(#1)}#3}
\Crefformat{equation}{#2\textup{(#1)}#3}

\crefrangeformat{equation}
{#3\textup{(#1)}#4--#5\textup{(#2)}#6}
\Crefrangeformat{equation}
{#3\textup{(#1)}#4--#5\textup{(#2)}#6}

\begin{document}
	
	\begin{frontmatter}
		\title{Berry--Esseen bounds for stochastic approximation with Markovian noise}
		\runtitle{Berry--Esseen bounds for Markovian stochastic approximation}
		
		\begin{aug}
			\author[A,B]{\fnms{Zhi-Jun}~\snm{Cai}\ead[label=e1]{zhijuncai@cuhk.edu.hk}}
			\author[A]{\fnms{Yu-Jie}~\snm{Qin}\ead[label=e2]{12631199@mail.sustech.edu.cn}}
			\author[A,C]{\fnms{Zhuo-Song}~\snm{Zhang}\ead[label=e3]{zhangzs3@sustech.edu.cn}}
			\address[A]{Department of Statistics and Data Science, Southern University of Science and Technology, Shenzhen 518055, China\printead[presep={,\ }]{e2}}
			\address[B]{Department of Statistics and Data Science, The Chinese University of Hong Kong, Shatin, New Territories, Hong Kong, China\printead[presep={,\ }]{e1}}
			\address[C]{National Center for Applied Mathematics Shenzhen, Shenzhen 518055, China\printead[presep={,\ }]{e3}}
		\end{aug}
		
		\begin{abstract}
			This paper studies the convergence rate of the Polyak-Ruppert average estimator in one-dimensional linear stochastic approximation recursion driven by Markov chains, allowing the Markov chain to start from any initial distribution. Under the conditions of positive Harris recurrence, non-periodicity, and one-step domination, combined with the fourth-moment condition of the coefficient function with respect to the stationary distribution, this paper uses the exchangeable pair technique to prove that after deterministic bias correction for the average estimator, a Berry--Esseen bound of order $n^{-\alpha/2}$ can be obtained, where the step size exponent is $\alpha \in (1/2, 1)$. Moreover, for the statistics without bias correction and after variance normalization, the asymptotic normal approximation error is asymptotically equivalent to a normal constant multiplied by $n^{1/2 - \alpha}$.
		\end{abstract}
		
		\begin{keyword}[class=MSC]
			\kwdgroup[type=primary]{\kwd{60F05}\kwd{60J10}}
			\kwdgroup[type=secondary]{\kwd{62L20}}
		\end{keyword}
		
		\begin{keyword}
			\kwd{Berry--Esseen bound}
			\kwd{stochastic approximation}
			\kwd{Polyak--Ruppert averaging}
			\kwd{Markov chain}
			\kwd{Stein's method}
		\end{keyword}
	\end{frontmatter}
	
	\section{Introduction}\label{sec:introduction}
	Stochastic approximation is a class of recursive algorithms that gradually update parameters based on noisy observations. When the target function cannot be directly calculated and can only be understood through observations with random errors, the algorithm uses continuously arriving data to gradually search for the zero points or maximum points of the function. The root-finding method for random equations proposed by Robbins and Monro is an early representative in this field~\citep{robbins1951stochastic}. Since then, stochastic approximation has been widely used in stochastic optimization, onlkine estimation, and adaptive control problems~\citep{kushner2003stochastic, borkar2008stochastic}. 
	
	From the perspective of the algorithm form, stochastic approximation updates parameters at each step based on the current parameters and the newly obtained observations. For recursive algorithms aimed at solving equations, the common update form is
	\[
	\theta_{k+1}
	=
	\theta_k-\epsilon_{k+1}F_{k+1}(\theta_k),
	\]
	where $F_{k+1}(\theta_k)$ represents the random update term obtained from the $(k+1)$th observation, and $\epsilon_{k+1}$ represents the step size. Since the update term contains random errors, the algorithm gradually reduces the impact of these errors on parameter estimation through multiple iterations. If these updates are examined from the perspective of long-term average behavior, the target corresponding to the recursion can usually be expressed as
	\[
	\int F(\theta,x)\,\pi(\mathrm{d}x)=0,
	\]
	here $\pi$ represents the reference distribution of the observed variable. This equation describes the average behavior of the random update terms and is a commonly used tool for analyzing recursive stability, convergence, and asymptotic distribution~\citep{kushner2003stochastic, borkar2008stochastic}.
	To further reduce the impact of recursive errors on the final estimator, the Polyak--Ruppert average can be applied to multiple iterative results, that is
	\begin{align}\label{eq-in-02}
		\bar\theta_n=\frac1n\sum_{k=1}^{n}\theta_k.
	\end{align}
	This method forms the final estimator by utilizing multiple iterative results. Under appropriate conditions, it can improve its asymptotic properties and achieve asymptotic efficiency~\citep{ruppert1988efficient,polyak1990new,polyak1992acceleration}. Existing studies have further provided finite sample errors, concentration inequalities, and high-probability bounds for linear stochastic approximation and its application in temporal difference learning~\citep{bhandari2021finite,mou2020linear,patil2023finite,samsonov2024improved,durmus2025finite}. These results mainly control the degree to which the estimator deviates from the target parameter. To study the distribution of the average estimator and conduct statistical inference, it is necessary to quantitatively analyze its normal approximation error.
	
	In the case of independent observations, the error of stochastic approximation mainly comes from the fluctuations of single-step random updates. By imposing appropriate moment conditions on the update terms and choosing gradually decreasing step sizes, the convergence and asymptotic distribution of the iterative sequence can be analyzed. 
	
	When observations are generated sequentially according to a Markov chain, observations at different times usually have correlations, and the current update is also influenced by previous states. Therefore, error analysis cannot only focus on single-step noise, but also needs to handle the time dependence between observations~\citep{fort2015central,bolthausen1980berry}. Existing studies mainly examine the finite sample errors, concentration inequalities, and high-probability bounds of such algorithms~\citep{bhandari2021finite,mou2020linear,patil2023finite,samsonov2024improved,durmus2025finite}. The policy evaluation in reinforcement learning can illustrate the specific sources of this Markov dependence. The agent interacts with the environment according to a fixed policy, observes the state at each time step, receives immediate rewards, and transitions to the next state. States and rewards are generated along the same trajectory, so adjacent observations are usually not independent. Temporal difference learning updates parameters using the value estimates of the current state and the next state. Under the linear value function approximation, this update can be expressed as a Markov chain-driven linear stochastic approximation~\citep{sutton1988learning,tsitsiklis1997analysis,bhandari2021finite}
	
	To illustrate the problem under study, let $S$ be the state space, $F(\theta,x)$ be the drift function depending on the observed state, and $\pi$ be the invariant distribution of the observation process. The target parameter $\theta^*$ satisfies the average equation
	\[
	\int_S F(\theta^*,x)\,\pi(\mathrm{d}x)=0.
	\]
	This paper considers the one-dimensional linear case $F(\theta,x) = a(x)\theta - b(x)$, for $k\ge 1$,
	\begin{equation}
		\theta_k=\theta_{k-1}
		-\epsilon_k\bigl\{a(X_k)\theta_{k-1}-b(X_k)\bigr\}, \quad \epsilon_k = \epsilon_0k^{-\alpha} \quad \alpha \in (1/2, 1).
		\label{eq-in-01}
	\end{equation}
	Let $\bar a = \int_S a(x)\,\pi(\mathrm{d}x)$ and $\bar b = \int_S b(x)\,\pi(\mathrm{d}x)$. When $\bar a \ne 0$, the unique solution to the average equation is $\theta^* = \bar b / \bar a$. This paper focuses on the finite sample normal approximation error of $\sqrt{n}(\bar\theta_n - \theta^*)$ when the observations are generated by a Markov chain with an invariant distribution of $\pi$ and the chain allows starting from any initial distribution.
	
	Under the conditions of independent observations and bounded coefficients, \citet{butyrin2025improved} utilized the finite sample covariance of the dominant linear term to obtain a degree of $O(n^{-\alpha/2}) $. Under the conditions of bounded coefficients and uniform geometric ergodicity, \citet{samsonov2025inference} established a non-asymptotic Berry--Esseen bound for the one-dimensional projection of the Polyak-Ruppert average. The normal approximation error is $O(n^{-1/4})$, which paper also studies the finite sample properties of the multiplier-block bootstrap confidence intervals under Markov sampling, and points out that the error order of the corresponding asymptotic variance estimation is $O(n^{-1/8})$, allowing for a logarithmic factor. For the additive functionals of Markov chains, \citet{srikant2025rates} obtained it under the Wasserstein distance $O\! \big(n^{-1/6}\sqrt{\log n}\big)$. The rate of normal approximation. For temporal difference learning, \citet{wu2025uncertainty} established a quantitative normal approximation result under the convex set distance, when $\alpha = 3/4$, the resulting error order is $O\! \big(n^{-1/4}\log n\big)$.
	
	To improve the convergence rate and achieve more accurate distribution approximation, Our proof uses the exchangeable-pair approach to Stein's method introduced by \citet{stein1986approximate}; see also \citet{shao2019berry} and \citet{zhang2022berry} for normal approximation with unbounded exchangeable pairs. We replace a local block of the Markov path by an independent copy from its conditional bridge law. Poisson decompositions control the recursive
	remainder and identify its deterministic bias.

	Let $\Phi$ denote the standard normal distribution function, and $\widetilde\sigma_n$ represent the exact finite-sample theoretical standard deviation of the dominant weighting term. $\beta_n$ is the first-order deterministic bias correction term determined by the overall quantity of the model. Definition
	\[
	T_n= \frac{\sqrt n\,(\bar\theta_n-\theta^*-\beta_n)}
	{\widetilde\sigma_n},
	\]
	when the observation process is positive Harris recurrent and non-periodic, and satisfies that the coefficient functions $a$ and $b$ have fourth moments with respect to the invariant distribution $\pi$, $\bar a > 0$ and $\alpha \in (1/2, 1)$, the statistic $T_n$ after first-order deterministic bias correction and standardization by $\widetilde\sigma_n$ satisfies
	\begin{equation*}
		\sup_{z\in\mathbb R}
		\bigl|
		\mathbb P(T_n\le z)-\Phi(z)
		\bigr|
		\le C n^{-\alpha/2}.
	\end{equation*}	
	
	In order to highlight the impact of this correction item on the normal approximation, we further examined the uncorrected statistics in \cref{cor:uncorrected}. This inference indicates that when the first-order bias coefficient is non-zero, the normal approximation error of the standardized average estimator has an order of $n^{1/2-\alpha}$; when the first-order bias disappears, the error order reverts to $O(n^{-\alpha/2})$.
	
	The rest of this paper is organized as follows. The main results are stated in \cref{sec:main-results}, the proofs of the main results are given in \cref{sec:proof-completion}, and the remaining proofs are deferred to the appendices.
	
	\section{Main results}\label{sec:main-results}
	
	Let $(S, \mathcal{S})$ be a standard Borel space, and $(X_k)_{k\ge0}$ be a homogeneous Markov chain defined on this space. The transition kernel and invariant probability measure of this chain are denoted by $P$ and $\pi$ respectively. The initial distribution is denoted as $\nu$. For a measurable function $f$,  put
	\[
	Pf(x)=\int_S f(y)P(x,\mathrm{d}y),
	\qquad
	P^t f(x)=\int_S f(y)P^t(x,\mathrm{d}y).
	\]
	
	The following introduces the basic conditions satisfied by the state process and the coefficient function.
	
	\begin{assumption}\label{ass:density-mixing}
		Assume that the Markov chain generated by $P$ is a non-periodic, positive Harris recurrent chain, and there exists a constant $K < \infty$ such that
		\begin{equation}\label{eq-ma-01}
			P(x,A)\le K\pi(A),
			\qquad x\in S,\ A\in\mathcal S.
		\end{equation}
	\end{assumption}
	
	\begin{remark}\label[remark]{rem:density-mixing}
		The finiteness of the state space is not essential. More generally, suppose
		that $P(x,\cdot)\ll\pi$ for every $x\in S$ and that
		\[
		p(x,y)=\frac{dP(x,\cdot)}{d\pi}(y),
		\qquad
		\sup_{x\in S}\operatorname*{ess\,sup}_{y\in S}p(x,y)\le K<\infty.
		\]
		Then $P(x,A)\le K\pi(A)$ for all $x\in S$ and $A\in\mathcal S$.
		Consequently, Assumption~\ref{ass:density-mixing} also applies to
		infinite standard Borel state spaces whenever the chain is positive Harris
		recurrent and aperiodic.
	\end{remark}
	
	\begin{assumption}\label[assumption]{ass:fourth-moments}
		The measurable functions $a,b:S\to\mathbb R$ satisfy
		\[
		\int_{S}\big[\abs{a(x)}^4+\abs{b(x)}^4\big]\pi(dx)
		\le M_{4}<\infty.
		\]
		Write $\bar a=\int_{S}a(x)\pi(dx)$ and assume $\bar a>0$.
	\end{assumption}

	\begin{remark}\label[remark]{rem:conditional-fourth}
		By \cref{eq-ma-01,ass:fourth-moments}, it follow that
		\begin{equation}
		\sup_{x\in S}P\big(\abs{a(x)}^4+\abs{b(x)}^4\big)(x)
			\le KM_{4}.
			\label{eq-ma-06}
		\end{equation}
	\end{remark}
	
	Under the step size and initial value conditions of \cref{eq-in-01}, let's denote $\bar b=\int_S b(x)\,\pi(\mathrm{d}x)$ and $\theta^*=\frac{\bar b}{\bar a}$. And define $h(x)=a(x)-\bar a$ and $\delta(x)=a(x)\theta^*-b(x)$.
	Put $\bar\theta_{n}=n^{-1}\sum_{k=1}^n\theta_{k}$. For $1\le i\le n$, define
	\begin{equation}\label{eq-mb-01}
		q_{i,n}=
		\begin{cases}
			\displaystyle
			\epsilon_{0}\sum_{j=1}^{n}
			\prod_{k=1}^{j}
			\bigl(1-\bar a\epsilon_{k}\bigr),
			& i=0,\\[3mm]
			\displaystyle
			\epsilon_{i}\sum_{j=i}^{n}
			\prod_{k=i+1}^{j}
			\bigl(1-\bar a\epsilon_{k}\bigr),
			& 1\le i\le n-1,\\[3mm]
			\epsilon_{n},
			& i=n.
		\end{cases}
	\end{equation}
	Under the stationary distribution, the long-term variance of $\delta(X_k)$ is denoted as
	\begin{equation}
		\Gamma
		=\int_{S}\delta(x)^2\pi(dx)
		+2\sum_{t=1}^{\infty}\int_{S}\delta(x)P^t\delta(x)\pi(dx),
		\qquad
		\sigma_{\infty}^2=\frac{\Gamma}{\bar a^2},
		\label{eq-mb-04}
	\end{equation}
	consequently, the precise finite sample variance of the dominant weighted term is defined as
	\begin{equation}
		\begin{aligned}
		\widetilde\sigma_{n}^2
		={}&\frac1n\sum_{i=1}^n q_i^2\int_S\delta(x)^2\pi(dx)
		+\frac2n\sum_{1\le i<j\le n}q_iq_j
		\int_S\delta(x)P^{j-i}\delta(x)\pi(dx).
		\end{aligned}
		\label{eq-mb-05}
	\end{equation}
	
	To depict the deviation resulting from the combined effect of Markov correlation and random coefficients, put
	\begin{align}
		&v_{h}(x)=\sum_{t=1}^{\infty}P^th(x),
		\qquad
		\kappa
		=\int_S\delta(x)v_h(x)\pi(dx)
		=\sum_{t=1}^{\infty}\int_S\delta(x)P^th(x)\pi(dx),
		\label{eq-mb-06}\\
		&A_{n}=\sum_{i=1}^{n-1}q_{i}\epsilon_{i},
		\qquad
		\beta_{n}=\frac{\kappa A_{n}}{n},
		\label{eq-mb-07}
	\end{align}
	where $\beta_n$ is the first-order deterministic bias correction term of the average estimator.
	
	The above definition gives the normalized and bias correction quantities used in the main theorem. We have the following theorem.
	\begin{theorem}\label[theorem]{thm:corrected}
		Suppose \cref{ass:density-mixing,ass:fourth-moments} hold and $\Gamma>0$.
		Then there exist constants $C,n_{0}<\infty$, independent of $n$ and the
		initial distribution $\nu$, such that, for every $n\ge n_{0}$,
		\begin{equation}
			\sup_{z\in\mathbb R}
			\bigg|
			\mathbb P \bigg(
			\frac{\sqrt n(\bar\theta_{n}-\theta^*-\beta_{n})}
			{\widetilde\sigma_{n}}
			\le z
			\bigg)
			-\Phi(z)
			\bigg|
			\le Cn^{-\alpha/2}.
			\label{eq-mc-03}
		\end{equation}
	\end{theorem}
	
	Centering at $\theta^*$ does not imply unbiasedness for a finite sample. In fact, according to \cref{eq-in-01},
	\[
	\E(\theta_k-\theta^*)
	=\E(\theta_{k-1}-\theta^*)
	-\epsilon_k\E\{a(X_k)(\theta_{k-1}-\theta^*)+\delta(X_k)\}.
	\]
	Even though $\int_S\delta(x)\pi(dx)=0$, the right side will still be affected by the initial value, transient distribution, as well as the correlation between $a(X_k)$ and $\theta_{k-1}$, thus generally $\E\bar\theta_n \ne \theta^*$. The $\beta_n$ in \cref{thm:corrected} is precisely used to eliminate the first-order Markov multiplicative noise bias; the following inference characterizes the approximation speed without making this correction.
	
	\begin{corollary}\label[corollary]{cor:uncorrected}
		Under the assumptions of \cref{thm:corrected}, 
		\begin{equation}
		\sup_{z\in\mathbb R}
		\bigg |\mathbb P\bigg(
		\frac{\sqrt n(\bar\theta_{n}-\theta^*)}
		{\widetilde\sigma_{n}}
		\le z\bigg)-\Phi(z)\bigg|
		\le C n^{1/2-\alpha}.
		\label{eq-mc-04}
		\end{equation}
		If $\kappa=0$, this upper bound improves to $C n^{-\alpha/2}$.
		If $\kappa\ne0$, then
		\begin{equation}
			\sup_{z\in\mathbb R}
			\bigg |\mathbb P\bigg(
			\frac{\sqrt n(\bar\theta_{n}-\theta^*)}
			{\widetilde\sigma_{n}}
			\le z\bigg)-\Phi(z)\bigg|
			=C_{\kappa}
			n^{1/2-\alpha}+o\big(n^{1/2-\alpha}\big), \quad C_{\kappa}=\frac{\abs{\kappa}\epsilon_{0}}
			{\sqrt{2\pi}\,\bar a(1-\alpha)\sigma_{\infty}}.
			\label{eq-mc-05}
		\end{equation}
		In particular, there are $c>0$ such that
		\begin{equation}
			\sup_{z\in\mathbb R}
			\bigg |\mathbb P\Big(
			\frac{\sqrt n(\bar\theta_{n}-\theta^*)}
			{\widetilde\sigma_{n}} \Big)\bigg|\ge cn^{1/2-\alpha}.
			\label{eq-mc-06}
		\end{equation}
	\end{corollary}
	
	\begin{remark}\label[remark]{rem:bias-vanishes}
		If the Markov chain starts from the invariant distribution $\pi$ and the observation sequence $(X_k)_{k\ge0}$ is mutually independent, then $P^th = 0$, which implies $v_h = 0$ and $\kappa = 0$. If $a$ is a constant, the same result holds, i.e., $\kappa = 0$. Therefore, the slower speed of $n^{1/2 - \alpha}$ indeed stems from the combined effect of Markov correlation and multiplicative noise.
	\end{remark}
	
\section{Proofs of main results}\label{sec:proof-completion}
In this section, we will prove \cref{thm:corrected} and \cref{cor:uncorrected}. First, we outline the proof strategy of \cref{thm:corrected} and break it down into five manageable parts; then we present some important auxiliary lemmas and demonstrate our proof and calculation process; finally, we prove \cref{cor:uncorrected}.

\subsection{Main idea of the proof of  \cref{thm:corrected}}\label{subsec:statistical-decomposition}
For each $k\geq 0$, let $\Delta_{k}=\theta_{k}-\theta^{*}$. It follows from \eqref{eq-in-01} that 
	\begin{equation}\label{eq-ka-02}
		\begin{aligned}
		\Delta_k
		&=\bigl(1-\bar a\epsilon_k\bigr)\Delta_{k-1}
		-\epsilon_k\big[(a(X_k)\theta^*-b(X_k)\big]
		-\epsilon_k\bigl(a(X_k)-\bar a\bigr)\Delta_{k-1}\\
		&=\Delta_0\prod_{r=1}^{k}\bigl(1-\bar a\epsilon_r\bigr)
		-\sum_{i=1}^{k}\epsilon_i
		\prod_{r=i+1}^{k}\bigl(1-\bar a\epsilon_r\bigr)
		\bigl[a(X_i)\theta^*-b(X_i)\bigr]\\
		&\quad
		-\sum_{i=1}^{k}\epsilon_i
		\prod_{r=i+1}^{k}\bigl(1-\bar a\epsilon_r\bigr)
		\bigl(a(X_i)-\bar a\bigr)\Delta_{i-1},
		\qquad 1\le k\le n.	
	    \end{aligned}
	\end{equation}
	
	Using \cref{eq-in-02}, we obtain
	\begin{align}
		\sqrt n(\bar\theta_n-\theta^*)
		={}&-\frac1{\sqrt n}\sum_{i=1}^n q_i\delta(X_i)
		+\frac{q_0\Delta_0}{\sqrt n\,\epsilon_0}
		-\frac1{\sqrt n}\sum_{i=1}^n
		q_i\bigl(a(X_i)-\bar a\bigr)\Delta_{i-1}.
		\label{eq-pj-02}
	\end{align}
    On the right side, in sequence, are the weighted additive terms, the initial value terms, and the adaptive terms. To standardize the main term, let
    \begin{align*}
    	&\mu_{\nu,n}=\frac1n\sum_{i=1}^nq_i\int_S\delta(x)(\nu P^i)(dx),\\
    	&\sigma_{\nu,n}^2=\frac1n\bigg\{
    	\sum_{i=1}^nq_i^2\int_S\delta(x)^2(\nu P^i)(dx)
    	+2\sum_{1\le i<j\le n}q_iq_j
    	\int_S\delta(x)P^{j-i}\delta(x)(\nu P^i)(dx)
    	-n^2\mu_{\nu,n}^2\bigg\}.
    \end{align*}
    Define
    \begin{align}
    	W &=-\frac{1}{\sigma_{\nu,n}\sqrt{n}}
    	\bigg\{\sum_{i=1}^nq_{i}
    	\bigg(\delta(X_{i})-\int_S\delta(x)(\nu P^i)(dx)\bigg)\bigg\}.
    	\label{eq-pa-01}\\
    	U_{c}
    	&=\frac{q_{0}\Delta_{0}}{\sigma_{\nu,n}\sqrt{n}\epsilon_{0}}
    	-\frac{1}{\sigma_{\nu,n}\sqrt{n}}
    	\sum_{i=1}^nq_{i} \bigl(a(X_i)-\bar a\bigr)\Delta_{i-1}
    	-\frac{\kappa A_{n}+\mu_{\nu,n}}{\sigma_{\nu,n}\sqrt{n}}.
    	\label{eq-pa-02}
    \end{align}
	Then \cref{eq-pj-02} gives
	\begin{equation}
		W+U_{c}
		=\frac{\sqrt{n}\big\{\bar\theta_{n}-\theta^*-\beta_{n}\big\}}
		{\sigma_{\nu,n}}.
		\label{eq-pa-03}
	\end{equation}
	Based on the above definition, we can obtain
	\begin{equation}
		\E W=0,
		\qquad
		\E W^2=1.
		\label{eq-pa-04}
	\end{equation}
	
	We next construct an exchangeable pair for the leading term in
	\cref{eq-pj-02}.
	
	Fix an integer $m_{0}\le m\le(n-3)/4$ and set $L=2m+1$. Let
	\[
	\mathcal F_{k}=\sigma(X_{0},\ldots,X_{k}),\qquad k\ge0,
	\qquad
	\mathcal F=\sigma(X_{0},X_{1},\ldots).
	\]
	For $1\le i\le n$, define
	\[
	\begin{aligned}
	\mathcal A_{i}
	&=\big\{j\in\{1,\ldots,n\}:\abs{j-i}\le m\big\},
	\qquad
	\mathcal B_{i}
	=\big\{j\in\{1,\ldots,n\}:\abs{j-i}\le2m\big\},\\
	c_{j}&=\#\big\{i:j\in\mathcal A_{i}\big\},
	\qquad
	w_{ij}=\frac{L}{c_{j}}\one_{\{j\in\mathcal A_{i}\}}.
	\end{aligned}
	\]
	
	Define $\mathcal F_{\ell_{i},r_{i}}
	=\sigma(X_{j}:1\le j\le n,\ j\notin\mathcal B_{i})$. After determining $\mathcal F_{\ell_{i},r_{i}}$, replicate the blocks independently from the same-condition bridge distribution to obtain $X_j^{(i)}$, and set $X_j^{(i)} = X_j$ outside of $\mathcal B_i$. Put
	\begin{equation}
		Y_{j}=-q_{j}\delta(X_{j}),
		\qquad
		Y_{j}^{(i)}=-q_{j}\delta(X_{j}^{(i)}).
		\label{eq-pb-03}
	\end{equation}
	Take an $I$ that is independent of all random quantities and uniformly distributed over $\{1,\ldots,n\}$. Let
	\[
	\mathbf X' = \mathbf X^{(I)},\qquad
	W'=\frac{\sum_{j=1}^nY_j^{(I)}+n\mu_{\nu,n}}
	{\sigma_{\nu,n}\sqrt n},\qquad \Xi=W-W'.
	\]
	Define
	\begin{align}
		D
		&=\frac{1}{\sigma_{\nu,n}\sqrt{n}}
		\sum_{j\in\mathcal A_{I}}w_{Ij}
		\big(Y_{j}-Y_{j}^{(I)}\big), \quad
		D^*
		=\frac{1}{\sigma_{\nu,n}\sqrt{n}}
		\sum_{j\in\mathcal A_{I}}w_{Ij}
		\big(\abs{Y_{j}}+\abs{Y_{j}^{(I)}}\big).
		\label{eq-pb-06}
	\end{align}
	Then $D^*\ge\abs D$. Starting from the same $\theta_{0}$, let
	$\Delta_{k}^{(i)}$ be the recursion driven by $(X_{j}^{(i)})$, and define
	$U_{c}'$ from \cref{eq-pa-02} with $(X_{j},\Delta_{j})$ replaced by
	$(X_{j}^{(I)},\Delta_{j}^{(I)})$. Finally, set
	\begin{equation}
		\widehat D=\E\big(D\mid\mathbf X,\mathbf X'\big),
		\qquad
		\widehat D^*=\E\big(D^*\mid\mathbf X,\mathbf X'\big).
		\label{eq-pb-08}
	\end{equation}

	Thus $(\mathbf X,\mathbf X')$ is exchangeable, and the versions in
	\cref{eq-pb-08} can be chosen so that $\widehat D$ is antisymmetric
	and $\widehat D^*$ is symmetric. Since $\abs{D}\le D^*$, conditional
	Jensen's inequality gives
	\[
	\abs{\widehat D}
	\le \E(\abs{D}\mid\mathbf X,\mathbf X')
	\le \widehat D^*.
	\]
	As $\Xi$ is measurable with respect to
	$\sigma(\mathbf X,\mathbf X')$, the tower property yields
	\begin{equation}
		\begin{aligned}
			\E(\widehat D\mid\mathbf X)
			&=\E(D\mid\mathbf X),\quad
			\E(\widehat D\Xi\mid\mathbf X)
			=\E(D\Xi\mid\mathbf X), \quad
			\E(\widehat D^*\Xi\mid\mathbf X)
			=\E(D^*\Xi\mid\mathbf X).
		\end{aligned}
		\label{eq-pb-10}
	\end{equation}
	Similarly, since $U_c-U_c'$ is measurable with respect to
	$\sigma(\mathbf X,\mathbf X')$,
	\begin{equation}
		\begin{aligned}
			\E\abs{\widehat D(U_c-U_c')}
			&=
			\E\!\big[
			\abs{\E(D\mid\mathbf X,\mathbf X')}
			\abs{U_c-U_c'}
			\big]
			\le \E\abs{D(U_c-U_c')}.
		\end{aligned}
		\label{eq-pb-11}
	\end{equation}
	
	Conditioning on the original path and then averaging over the bridge copy and the uniform index $I$, we obtain from \cref{eq-pb-06}
	\begin{align*}
		\E\big(D\mid\mathbf X\big)
		&=\frac{1}{n\sigma_{\nu,n}\sqrt n}
		\sum_{i=1}^n\sum_{j\in\mathcal A_{i}}w_{ij}
		\big\{Y_{j}-\E\big(Y_{j}^{(i)}\mid\mathcal F_{\ell_{i},r_{i}}\big)\big\}\\
		&=\frac{1}{n\sigma_{\nu,n}\sqrt n}
		\bigg\{
		\sum_{j=1}^nY_{j}\sum_{i=1}^nw_{ij}
		-\sum_{i=1}^n\sum_{j\in\mathcal A_{i}}w_{ij}
		\E\big(Y_{j}^{(i)}\mid\mathcal F_{\ell_{i},r_{i}}\big)
		\bigg\}\\
		&=\frac{L}{n}\bigg\{
		W-\frac{\mu_{\nu,n}}{\sigma_{\nu,n}\sqrt n}
		-\frac{1}{L\sigma_{\nu,n}\sqrt n}
		\sum_{i=1}^n\sum_{j\in\mathcal A_{i}}w_{ij}
		\E\big(Y_{j}^{(i)}\mid\mathcal F_{\ell_{i},r_{i}}\big)
		\bigg\},
	\end{align*}
	where the last equality follows from $\sum_{i=1}^nw_{ij}=L$ and the
	definitions of $Y_{j}$ and $W$. Set
	\begin{equation}
		\lambda=\frac{L}{n},
		\qquad
		R=-\frac{1}{L\sigma_{\nu,n}\sqrt n}
		\sum_{i=1}^n\sum_{j\in\mathcal A_{i}}w_{ij}
		\E\big(Y_{j}^{(i)}\mid\mathcal F_{\ell_{i},r_{i}}\big)
		-\frac{\mu_{\nu,n}}{\sigma_{\nu,n}\sqrt n}.
		\label{eq-pc-03}
	\end{equation}
	Therefore,
	\begin{equation}
		\E\big(D\mid\mathbf X\big)=\lambda(W+R).
		\label{eq-pc-02}
	\end{equation}
	
	When $\E W = 0$, $\E W^2 = 1$, $\widehat D$ is antisymmetric,
	$\widehat D^*$ is symmetric and $\widehat D^* \ge \abs{\widehat D}$,
	the exchange inequality in \citet[Corollary~4.3]{zhang2022berry} is
	\begin{equation}\label{eq-pb-07}
		\begin{aligned}
			\sup_{z\in\mathbb R}
			\big|
			\mathbb P \big(W+U_{c}\le z\big)-\Phi(z)
			\big|
			\le{}&
			\E \big| 1-\frac1{2\lambda}\E(\widehat D\Xi\mid\mathbf X)\big|
			+\frac1\lambda
			\E\big|\E(\widehat D^*\Xi\mid\mathbf X)\big|
			+\E\abs R\\
			&+\frac3{2\lambda}
			\E\big|\widehat D(U_{c}-U_{c}')\big|
			+\E\abs U_{c} .
		\end{aligned}
	\end{equation}
	In the order in which they occur on the right-hand
	side of \cref{eq-pb-07}, write
	\begin{align}
		J_{1}&=\E \big| 1-\frac1{2\lambda}\E(\widehat D\Xi\mid\mathbf X)\big|,
		&J_{2}&=\frac{1}{\lambda}
		\E\big|\E(\widehat D^*\Xi\mid\mathbf X)v,
		\notag\\
		J_{3}&=\E\abs{R},
		&J_{4}&=\frac{3}{2\lambda}
		\E\big|\widehat D(U_{c}-U_{c}')\big|,\qquad
		J_{5}=\E\abs{U_{c}},
		\label{eq-pb-12}
	\end{align}
	Here, $J_1$ and $J_2$ represent the projection errors of the local quadratic terms of the conditional bridge, $J_3$ is the regression residual term, $J_4$ is the influence of the bridge substitution on the corrected residual term, and $J_5$ is the magnitude of the corrected residual term itself. 
		
\subsection{Preliminary lemmas}\label{subsec:preliminary-lemmas}
	In this section, we will present some auxiliary lemmas that are used to prove the results. The \cref{lem:oracle-variance,lem:bridge-kernel} support bridge resampling and covariance estimation, \cref{lem:projection-sum} controls $J_1$ and $J_2$, \cref{lem:corrected-remainder} controls $J_4$ and $J_5$, and \cref{lem:An-asymptotic} is used for \cref{cor:uncorrected} in \cref{subsec:corollary-proofs}.
		
	The following lemma provides a covariance bound used to estimate
	$J_1$ and prove \cref{lem:oracle-variance}, and a weighted-sum bound
	used to estimate $J_4$ and prove \cref{lem:corrected-remainder}.
	Its proof is given in \cref{app:common-tools}.
	\begin{lemma}\label[lemma]{lem:markov-toolkit}
		Suppose \cref{ass:density-mixing} holds and $X_0\sim\nu$.
		\begin{enumerate}[label=(\roman*)]
			\item For $r\ge0$ and $t\ge1$, let $\varphi$ be
			$\mathcal F_r$-measurable and $\psi$ be measurable with respect to $\sigma(X_{r+t},X_{r+t+1},\ldots)$. If
			$\E\abs{\varphi}<\infty$ and
			$\int_S\E_x\abs{\psi}\,\pi(dx)<\infty$, then
			\begin{equation}
				\abs{\cov(\varphi,\psi)}
				\le C\rho^t\E\abs{\varphi}
				\int_S\E_x\abs{\psi}\,\pi(dx).
				\label{eq-pe-01}
			\end{equation}
			
			\item If $f\in L^2(\pi)$ is centered, then, for every
			finitely supported deterministic sequence $(c_i)_{i\ge1}$,
			\begin{equation}
				\E\Big(\sum_{i\ge1}c_i f(X_i)\Big)^2
				\le C_f\bigg\{
				\sum_{i\ge1}c_i^2+
				\big(\sum_{i\ge1}\abs{c_i}\rho^i\big)^2
				\bigg\}.
				\label{eq-pe-02}
			\end{equation}
		\end{enumerate}
	\end{lemma}

	This lemma is used to compare the means and variances of the dominant weighted terms under different initial distributions and ensure that the standardized quantities are non-degenerate; its proof is given in \cref{app:common-tools}.
	
	\begin{lemma}\label[lemma]{lem:oracle-variance}
		Assuming $\Gamma > 0$. Then the finite sample variance of the main term in the stationary reference scenario satisfies
		\begin{equation}
			\widetilde\sigma_{n}^{\,2}
			\longrightarrow
			\frac{\Gamma}{\bar a^{\,2}}.
			\label{eq-pi-01}
		\end{equation}
		
		Furthermore, there exists a constant $C < \infty$, independent of both $n$ and the initial distribution $\nu$, such that for any initial distribution $\nu$,
		\begin{equation}
			\abs{\mu_{\nu,n}}\le C,
			\qquad
			\abs{\sigma_{\nu,n}^{2}-\widetilde\sigma_{n}^{\,2}}
			\le Cn^{-1}.
			\label{eq-pi-04}
		\end{equation}
	\end{lemma}
	
	The bias expansion in the next lemma is used in the proof of
	\cref{cor:uncorrected}. Its proof is given in \cref{app:common-tools}
	
	\begin{lemma}\label[lemma]{lem:An-asymptotic}
		As $n\to\infty$,
		\begin{equation*}
			A_{n}\sim\frac{\epsilon_{0}}{\bar a(1-\alpha)}n^{1-\alpha},
			\qquad
			\beta_{n}=\frac{\kappa\epsilon_{0}}{\bar a(1-\alpha)}n^{-\alpha}
			+o\big(n^{-\alpha}\big).
		\end{equation*}
	\end{lemma}
	
	From \cref{rem:density-mixing}, we can take a transfer density $p_t(x,y)$ that is jointly measurable with respect to $(x,y)$, such that
	\[
	P^t(x,A)=\int_A p_t(x,y)\,\pi(dy),
	\qquad
	\abs{p_t(x,y)-1}\le C\rho^t,
	\]
	where $t \ge 1$, $x \in S$, and the inequality holds for $\pi$-almost-everywhere $y$. For $j \ge 1$, let
	\[
	d_j(y)=\int_S p_j(x,y)\,\nu(dx).
	\]
	Then $d_j$ is the density of $X_j$ with respect to $\pi$, and it satisfies
	\[
	\abs{d_j(y)-1}\le C\rho^j,
	\qquad
	0\le d_j(y)\le C.
	\]
	Therefore, the integrable moment estimation under $\pi$ can be uniformly used to control $X_j$ ($j \ge 1$).
	
	Fix $n \ge 1$ and $1 \le \ell \le r \le n$, and let
	\[
	[\ell,r]=\{\ell,\ldots,r\},
	\qquad
	\mathcal F_{\ell,r}
	=
	\sigma\bigl(X_s:1\le s\le n,\ s\notin[\ell,r]\bigr),
	\]
	when $\ell > 1$, the left boundary is $X_{\ell - 1}$; when $r < n$, the right boundary is $X_{r + 1}$. If $\ell = 1$, then $X_0$ can be generated along with the block $(X_1, \ldots, X_r)$ during the bridge resampling.
	
	The next two lemmas respectively establish the conditional bridge estimation and the exchange property of local resampling, providing a foundation for the construction of the exchange pairs. Its proof is given in \cref{app:auxiliary-lemmas}.
	\begin{lemma}\label[lemma]{lem:bridge-kernel}
		Assuming that \cref{ass:density-mixing} holds. Then there exists a constant $C < \infty$ and an integer $m_{0} \ge 1$ such that for any $n \ge 1$ and $1 \le \ell \le r \le n$, the following conclusion is true.
		
		\begin{enumerate}[label=(\roman*),ref=(\roman*)]
			
			\item\label{item:bridge-density}
			For any $\ell \le j \le r$, given $\mathcal F_{\ell,r}$,
			the conditional density of $X_j$ with respect to $\pi$ is
			\[
			\begin{cases}
				\displaystyle
				\frac{
					p_{j-\ell+1}(x,z)\,
					p_{r-j+1}(z,y)
				}{
					p_{r-\ell+2}(x,y)
				},
				&
				\ell>1,\ r<n,\
				x=X_{\ell-1},\ y=X_{r+1},
				\\[3ex]
				\displaystyle
				\frac{
					d_j(z)\,
					p_{r-j+1}(z,y)
				}{
					d_{r+1}(y)
				},
				&
				\ell=1,\ r<n,\
				y=X_{r+1},
				\\[3ex]
				p_{j-\ell+1}(x,z),
				&
				\ell>1,\ r=n,\
				x=X_{\ell-1},
				\\[1ex]
				d_j(z),
				&
				\ell=1,\ r=n.
			\end{cases}
			\]
			The above formula holds for $\pi$-almost everywhere $z$ and for the joint distribution of the corresponding boundary variables almost everywhere.
			When $\ell > 1$, $r < n$, and
			$r - \ell + 2 \ge 2m_0$, the denominator
			$p_{r - \ell + 2}(x, y)$ has a uniform positive lower bound.
			
			\item\label{item:bridge-uniform-bound}
			The conditional densities in part~\ref{item:bridge-density} are uniformly bounded in $z$ by a constant $C$, for almost every admissible value of the boundary variables. The required length condition is
			\begin{align*}
				r-\ell+2&\ge 2m_{0},
				&&\text{if }\ell>1\text{ and }r<n,\\
				r+1&\ge m_{0},
				&&\text{if }\ell=1\text{ and }r<n.
			\end{align*}
			When $r=n$, no additional length condition is required, regardless of whether $\ell=1$ or $\ell>1$.
			
			\item\label{item:bridge-forgetting}
			Let $f \in L^1(\pi)$ satisfy $\int_S f(x)\,\pi(dx)=0$. Take $m \ge m_{0}$. If $\ell > 1$, assume
			$j - \ell + 1 \ge m$; if $r < n$, assume
			$r - j + 1 \ge m$. Then
			\[
			\big|
				\E\bigl[f(X_j)\mid\mathcal F_{\ell,r}\bigr]
			\big|
			\le
			C\bigl(
			\rho^m+\one_{\{\ell=1\}}\rho^j
			\bigr)
			\int_S\abs{f(x)}\,\pi(dx).
			\]
			Specifically, when $\ell = 1$ and $r = n$,
			\[
			\abs{\E f(X_j)}
			\le
			C\rho^j
			\int_S\abs{f(x)}\,\pi(dx).
			\]
			
			\item\label{item:bridge-realization}
			Given the existing boundary variables, the conditional distribution of the block
			$(X_{\ell},\ldots,X_r)$ can be generated measurably by an independent
			$\operatorname{Unif}(0,1)$ random variable.
			When $\ell=1$, $X_0$ can be generated along with this block,
			so that $(X_0,\ldots,X_r)$ has the conditional distribution given by
			$\mathcal F_{1,r}$.	
		\end{enumerate}
	\end{lemma}
	
	\begin{lemma}\label[lemma]{lem:exchangeable-construction}
		Under \cref{ass:density-mixing}, for the resampling construction in \cref{subsec:statistical-decomposition},
		write $\ell_i=\min\mathcal B_i$ and $r_i=\max\mathcal B_i$, and let
		\[
		\mathcal L_i
		=\mathcal B_i
		\cup\{\ell_i-1:\ell_i>1\}
		\cup\{r_i+1:r_i<n\}.
		\]
		Let $(I_i)_{i=1}^n$ be the mutually independent $\operatorname{U}(0,1)$ variables of $\mathcal F$.
		\begin{enumerate}[label=(\roman*),ref=(\roman*)]
			\item For each $i$, conditional on
			$\mathcal F_{\ell_{i},r_{i}}$, the vectors
			$X_{\mathcal B_{i}}$ and $X_{\mathcal B_{i}}^{(i)}$ are independent
			and have the same conditional distribution. When $\ell_{i}=1$,
			the same property holds for the augmented vectors
			$(X_{0},X_{\mathcal B_{i}})$ and
			$(X_{0}^{(i)},X_{\mathcal B_{i}}^{(i)})$. Conditional on $\mathcal F$,
			the resampled segments are independent over $i$, and the segment
			generated using $I_{i}$ is measurable with respect to
			$\sigma(X_{\mathcal L_{i}},I_{i})$. Moreover,
			\[
			(X_{0}^{(i)},X_{1}^{(i)},\ldots,X_{n}^{(i)})
			\stackrel{d}{=}
			(X_{0},X_{1},\ldots,X_{n}).
			\]
			
			\item If $r_{i}<n$ and $s=r_{i}+1$, then
			$\mathcal F_{s}\vee\sigma(I_{i})$ and
			$\sigma(X_{s+1},X_{s+2},\ldots)$ are conditionally independent given
			$X_{s}$. Hence, for every $f\in L^1(\pi)$ and $j\ge s+1$,
			\[
			\E\big\{f(X_{j})\mid
			\mathcal F_{j-1}\vee\sigma(I_{i})\big\}
			=Pf(X_{j-1}).
			\]
		\end{enumerate}
	\end{lemma}
	
	The next lemma controls the conditional projection sums arising in
	$J_1$ and $J_2$. Its proof is given in \cref{app:auxiliary-lemmas}.
	
	\begin{lemma}\label[lemma]{lem:projection-sum}
		Let $\mathcal G$ be a sub-$\sigma$-field of the underlying probability space,
		and let $Z_{1},\ldots,Z_{n}$ be square-integrable real-valued random variables.
		Suppose that there exist constants $C_{0}<\infty$, $\rho\in(0,1)$, and an
		integer $m\ge1$ such that, for all $1\le i,j\le n$,
		\[
		\abs{\cov(Z_{i},Z_{j})}
		\le
		C_{0}m^4
		\begin{cases}
			1, & \abs{i-j}\le 8m+4,\\[1mm]
			\rho^{\lfloor \abs{i-j}/2\rfloor},
			& \abs{i-j}>8m+4.
		\end{cases}
		\]
		Then there exists a constant $C<\infty$, depending only on $C_{0}$ and
		$\rho$, such that
		\[
		\Big\|
		\sum_{i=1}^n
		\big\{
		\E(Z_{i}\mid\mathcal G)-\E Z_{i}
		\big\}
		\Big\|_{2}
		\le
		C\sqrt{nm^5}.
		\]
		If $\E Z_{i}=0$ for every $i$, then the sum on the left-hand side reduces to $\sum_{i=1}^n \E(Z_{i}\mid\mathcal G)$.
	\end{lemma}
	
	The following lemma bounds $U_c$ and its change under one bridge
	replacement, which control $J_5$ and $J_4$, respectively.
	Its proof is given in
	\cref{app:corrected-remainder}.
	
	\begin{lemma}\label[lemma]{lem:corrected-remainder}
		Suppose $\Gamma>0$. Under $\Pnu$, for all sufficiently large $n$,
		\begin{equation}
			\E\abs{U_{c}}\le Cn^{-\alpha/2}.
			\label{eq-qc-02}
		\end{equation}
		Moreover,
		\begin{equation}
			\E\abs{U_{c}-U_{c}'}^2
			\le C\sigma_{\nu,n}^{-2}
			\bigg\{\frac{m^{2+1/\alpha}}{n^2}
			+\frac{m}{n^{1+\alpha}}\bigg\}.
			\label{eq-qh-03}
		\end{equation}
	\end{lemma}

\subsection{Proof of the main theorem}\label{subsec:main-theorem-proof}
	
	\begin{proof}[Proof of \cref{thm:corrected}]
		We first estimate $J_{1}$. Adding and subtracting
		$(2\lambda)^{-1}\E(\widehat D\Xi)$ in the definition of $J_{1}$ gives
		\begin{align}
			J_{1}
			\le{}&
			\frac{1}{2\lambda}
			\big\|
			\E(\widehat D\Xi\mid\mathbf X)
			-\E(\widehat D\Xi)
			\big\|_{2}
			+
			\big|
			1-\frac{1}{2\lambda}\E(\widehat D\Xi)
			\big|.
			\label{eq-j1-decomposition}
		\end{align}
		
		Conditional on $I=i$, put
		\[
		G_{i}=
		\big\{\sum_{j\in\mathcal A_{i}}
		w_{ij}\big(Y_{j}-Y_{j}^{(i)}\big)\big\}
		\big\{\sum_{k\in\mathcal B_{i}}
		\big(Y_{k}-Y_{k}^{(i)}\big)\big\}.
		\]
		Averaging over the uniform index $I$ yields
		\begin{equation}
			\E(\widehat D\Xi\mid\mathbf X)
			=
			\frac{1}{\sigma_{\nu,n}^2n^2}
			\sum_{i=1}^n\E(G_{i}\mid\mathbf X),
			\qquad
			\E(\widehat D\Xi)
			=
			\frac{1}{\sigma_{\nu,n}^2n^2}
			\sum_{i=1}^n\E G_{i}.
			\label{eq-j1-projection}
		\end{equation}
		Consequently,
		\begin{equation}
			\big\|
			\E(\widehat D\Xi\mid\mathbf X)-\E(\widehat D\Xi)
			\big\|_{2}
			=
			\frac{1}{\sigma_{\nu,n}^2n^2}
			\big\|
			\sum_{i=1}^n
			\big\{\E(G_{i}\mid\mathbf X)-\E G_{i}\big\}
			\big\|_{2}.
			\label{eq-j1-centered-projection}
		\end{equation}
		
		By construction, $G_i$ is measurable with respect to
		$\sigma(X_{\mathcal L_i},I_i)$. The uniform fourth-moment bounds,
		the boundedness of $w_{ij}$, and Minkowski's inequality give
		$\|G_i\|_2\le Cm^2$. Thus, when $\abs{i-j}\le8m+4$,
		\[
		\abs{\cov(G_i,G_j)}
		\le 2\|G_i\|_2\|G_j\|_2
		\le Cm^4.
		\]
		
		Suppose $i<j$ and $j-i>8m+4$. Conditional independence of
		$I_i$ and $I_j$ given $\mathcal F$ implies
		\[
		\cov(G_i,G_j)
		=\cov\big(
		\E(G_i\mid\mathcal F),
		\E(G_j\mid\mathcal F)
		\big).
		\]
		The two conditional expectations depend only on $X_{\mathcal L_i}$
		and $X_{\mathcal L_j}$, respectively. These sets are separated by at
		least $\lfloor(j-i)/2\rfloor$ time steps. Conditional Jensen's
		inequality and the uniform fourth-moment bounds give
		\[
		\E\big|\E(G_i\mid\mathcal F)\big|\le Cm^2,
		\qquad
		\int_S\E_x\big|\E(G_j\mid\mathcal F)\big|\,\pi(dx)\le Cm^2.
		\]
		Applying \cref{eq-pe-01} to these conditional expectations yields
		\begin{align*}
		\abs{\cov(G_i,G_j)}
		&\le C\rho^{\lfloor(j-i)/2\rfloor}
		\E\big|\E(G_i\mid\mathcal F)\big|
		\int_S\E_x\big|\E(G_j\mid\mathcal F)\big|\,\pi(dx)\\
		&\le Cm^4\rho^{\lfloor(j-i)/2\rfloor}.
		\end{align*}
		The case $j<i$ follows by symmetry. Combining this estimate with
		the bound for $\abs{i-j}\le8m+4$ gives
		\begin{equation}
		\abs{\cov(G_i,G_j)}
		\le Cm^4
		\begin{cases}
		1, & \abs{i-j}\le8m+4,\\
		\rho^{\lfloor\abs{i-j}/2\rfloor},
		& \abs{i-j}>8m+4.
		\end{cases}
		\label{eq-j1-covariance}
		\end{equation}
		Applying \cref{lem:projection-sum} with
		$\mathcal G=\sigma(\mathbf X)$ and $Z_{i}=G_{i}$ gives
		\begin{equation}
			\Big\|
			\sum_{i=1}^n
			\big\{\E(G_{i}\mid\mathbf X)-\E G_{i}\big\}
			\Big\|_{2}
			\le C\sqrt{nm^5}.
			\label{eq-j1-projection-bound}
		\end{equation}
		Using the lower bound for $\sigma_{\nu,n}$ in
		\cref{eq-j1-centered-projection}, we obtain
		\begin{equation}
			\big\|
			\E(\widehat D\Xi\mid\mathbf X)-\E(\widehat D\Xi)
			\big\|_{2}
			\le
			C\frac{m^{5/2}}{n^{3/2}}.
			\label{eq-j1-centered-bound}
		\end{equation}
		
		By exchangeability and the antisymmetry of $D$, then $\E(\widehat D\Xi)=\E(D\Xi)=2\E(DW)$. Using the regression identity in \cref{eq-pc-02}, $\E W^2=1$, and the
		$\mathbf X$-measurability of $W$, we have
		\[
		\E(\widehat D\Xi)
		=2\lambda\{1+\E(WR)\}.
		\]
		Therefore,
		\[
		\big|1-\frac{1}{2\lambda}\E(\widehat D\Xi)\big|
		=
		\mid \E(WR) \mid 
		\le \|W\|_{2}\|R\|_{2}.
		\]
		
		We next control the remainder in this identity. If $j\in\mathcal A_{i}$, then
		$j$ is at distance at least $m$ from every present boundary of
		$\mathcal B_{i}$. Since $\int_S\delta(x)\pi(dx)=0$, part~\ref{item:bridge-forgetting}
		of \cref{lem:bridge-kernel}, boundedness of $q_{j}$, and \cref{eq-pb-03} yield
		\begin{equation}
			\big|\E\big(Y_{j}^{(i)}\mid\mathcal F_{\ell_{i},r_{i}}\big)\big|
			\le C\big\{\rho^m+\one_{\{\ell_{i}=1\}}\rho^j\big\},
			\qquad j\in\mathcal A_{i}.
			\label{eq-pk-04}
		\end{equation}
		The extra term at the left endpoint satisfies
		\[
		\sum_{i:\ell_{i}=1}\sum_{j\in\mathcal A_{i}}w_{ij}\rho^j
		\le
		2\sum_{i=1}^{2m+1}\sum_{j=1}^{3m+1}\rho^j
		\le Cm,
		\]
		where the first inequality uses $w_{ij}\le2$, together with
		$\ell_{i}=1\Rightarrow i\le2m+1$ and
		$j\in\mathcal A_{i}\Rightarrow j\le i+m\le3m+1$.
		The total deterministic weight is
		$\sum_{i}\sum_{j\in\mathcal A_{i}}w_{ij}=Ln$. 
		
		Substituting these bounds into \cref{eq-pc-03}, and using $L=2m+1$,
		$\abs{\mu_{\nu,n}}\le C$, and the lower bound for $\sigma_{\nu,n}$, gives
		\begin{equation}
			\norm{R}_{2}\le C\big\{\sqrt n\,\rho^m+n^{-1/2}\big\}.
			\label{eq-pk-01}
		\end{equation}

		Since $\|W\|_{2}=1$, $\lambda=L/n$, $L=2m+1$, and
		\cref{eq-pk-01,eq-j1-centered-bound} hold, we conclude that
		\begin{align*}
			J_{1}
			&\le
			C\frac{n}{m}\frac{m^{5/2}}{n^{3/2}}
			+
			C\big\{\sqrt n\,\rho^m+n^{-1/2}\big\}\\
			&\le
			C\bigg\{
			\frac{m^{3/2}}{\sqrt n}
			+\sqrt n\,\rho^m
			+n^{-1/2}
			\bigg\}.
		\end{align*}
		
		We next estimate $J_{2}$. By \cref{eq-pb-10}, and conditional on $I=i$, the definitions of
		$D^*$ and $\Xi$ give
		\[
		D^*\Xi=\frac{1}{\sigma_{\nu,n}^2n}G_{i}^*, \quad G_{i}^*
		=
		\Big\{\sum_{j\in\mathcal A_{i}}
		w_{ij}\big(|Y_{j}|+|Y_{j}^{(i)}|\big)\Big\}
		\Big\{\sum_{k\in\mathcal B_{i}}
		\Big(Y_{k}-Y_{k}^{(i)}\big)\Big\}.
		\]
		The uniform averaging argument used in \cref{eq-j1-projection} gives
		\begin{equation}
			\E(\widehat D^*\Xi\mid\mathbf X)
			=
			\frac{1}{\sigma_{\nu,n}^2n^2}
			\sum_{i=1}^n\E(G_{i}^*\mid\mathbf X).
			\label{eq-j2-projection}
		\end{equation}
		
		Interchanging the original and copied bridges leaves the first factor in
		$G_{i}^*$ unchanged and changes the sign of the second factor. Hence $\E G_{i}^*=0$.
		
		The covariance bound in \cref{eq-j1-covariance} also holds with $G_{i}$
		replaced by $G_{i}^*$. Since $\E G_{i}^*=0$, \cref{lem:projection-sum} gives
		\begin{equation}
			\big\|
			\sum_{i=1}^n\E(G_{i}^*\mid\mathbf X)
			\big\|_{2}
			\le C\sqrt{nm^5}.
			\label{eq-j2-projection-bound}
		\end{equation}
		Combining \cref{eq-j2-projection} and
		\cref{eq-j2-projection-bound}, and using the lower bound for
		$\sigma_{\nu,n}$, we obtain
		\begin{equation}
			\big\|
			\E(\widehat D^*\Xi\mid\mathbf X)
			\big\|_{2}
			\le C\frac{m^{5/2}}{n^{3/2}}.
			\label{eq-j2-centered-bound}
		\end{equation}
		
		Therefore, using $\E|Z|\le\|Z\|_{2}$ and
		$\lambda=L/n\asymp m/n$,
		\begin{align*}
			J_{2}
			&\le
			\frac{1}{\lambda}
			\big\|
			\E(\widehat D^*\Xi\mid\mathbf X)
			\big\|_{2}
			\le
			C\frac{n}{m}\frac{m^{5/2}}{n^{3/2}}
			=
			C\frac{m^{3/2}}{\sqrt n}.
		\end{align*}
		
		For $J_{3}$, \cref{eq-pb-12,eq-pk-01} and Cauchy--Schwarz give
		\begin{align*}
			J_{3}=\E\abs{R}
			\le\norm{R}_{2}
			\le C\big\{\sqrt n\,\rho^m+n^{-1/2}\big\}.
		\end{align*}
		
		By the finite-path distribution identity in \cref{lem:exchangeable-construction}, the bridge-updated sum satisfies the same estimate, we can obtain
		\[
		\E\big|
		\sum_{j\in\mathcal A_i}w_{ij}q_j\delta(X_j^{(i)})
		\big|^2
		=
		\E\big|
		\sum_{j\in\mathcal A_i}w_{ij}q_j\delta(X_j)
		\big|^2
		\le Cm.
		\]
		Since $I$ is independent of the original path and each substitution path, combining \cref{eq-pb-06}, we have
		\[
		\begin{aligned}
			\E(D^2\mid I=i)
			&=
			\frac{1}{n\sigma_{\nu,n}^2}
			\E\Big|
			\sum_{j\in\mathcal A_i}w_{ij}q_j
			\{\delta(X_j^{(i)})-\delta(X_j)\}
			\Big|^2\\
			&\le
			\frac{2}{n\sigma_{\nu,n}^2}
			\bigg\{
			\E\Big|
			\sum_{j\in\mathcal A_i}w_{ij}q_j\delta(X_j^{(i)})
			\Big|^2
			+
			\E\Big|
			\sum_{j\in\mathcal A_i}w_{ij}q_j\delta(X_j)
			\Big|^2
			\bigg\}\\
			&\le \frac{Cm}{n\sigma_{\nu,n}^2}.
		\end{aligned}
		\]
		Averaging over the uniform random index $I$ gives
		\begin{equation}
			\E D^2
			=\frac1n\sum_{i=1}^n\E(D^2\mid I=i)
			\le \frac{Cm}{n\sigma_{\nu,n}^2}.
			\label{eq-qh-16}
		\end{equation}
		The projection inequality in \cref{eq-pb-11} and Cauchy--Schwarz give
		\begin{equation*}
			J_{4}
			\le\frac{3}{2\lambda}
			\big(\E D^2\big)^{1/2}
			\big\{\E\abs{U_{c}-U_{c}'}^2\big\}^{1/2}.
		\end{equation*}
		Using \cref{eq-qh-16,lem:corrected-remainder}, and the lower bound for $\sigma_{\nu,n}$, we obtain
		\begin{align*}
			J_{4}
			&\le C\frac{n}{m}\bigg(\frac{m}{n}\bigg)^{1/2}
			\bigg\{\frac{m^{2+1/\alpha}}{n^2}
			+\frac{m}{n^{1+\alpha}}\bigg\}^{1/2}\\
			&\le C\bigg\{\frac{m^{1/2+1/(2\alpha)}}{\sqrt n}
			+n^{-\alpha/2}\bigg\}
			\le C\bigg\{\frac{m^{3/2}}{\sqrt n}
			+n^{-\alpha/2}\bigg\}.
		\end{align*}
		
		Finally, the definition of $J_{5}$ and \cref{lem:corrected-remainder} give
		\begin{equation*}
			J_{5}=\E\abs{U_{c}}\le C_{\alpha} n^{-\alpha/2}.
		\end{equation*}
		
		Combining the five bounds gives
		\begin{equation}
			J_{1}+J_{2}+J_{3}+J_{4}+J_{5}
			\le C\bigg\{n^{-\alpha/2}+\frac{m^{3/2}}{\sqrt n}
			+\sqrt n\,\rho^m\bigg\}.
			\label{eq-mc-02}
		\end{equation}
		Since $\rho\in(0,1)$, choose
		$K>(1+\alpha)/\{2\log(1/\rho)\}$ and put
		$m_{n}=\lceil K\log n\rceil$. This choice is admissible for all sufficiently
		large $n$, and, because $\alpha<1$,
		\[
		\frac{m_{n}^{3/2}}{\sqrt n}=o\big(n^{-\alpha/2}\big),
		\qquad
		\sqrt n\,\rho^{m_{n}}\le Cn^{-\alpha/2}.
		\]
		Thus \cref{eq-pa-03,eq-mc-02} give
		\begin{equation*}
			\sup_{z\in\mathbb R}
			\Big|
			\mathbb P \big(
			\frac{\sqrt n(\bar\theta_{n}-\theta^*-\beta_{n})}{\sigma_{\nu,n}}
			\le z\big)
			-\Phi(z)
			\Big|
			\le C_{\alpha} n^{-\alpha/2}.
		\end{equation*}
		By \cref{lem:oracle-variance},
		$\sigma_{\nu,n}/\widetilde\sigma_{n}=1+O(n^{-1})$ uniformly in $\nu$. For
		$c>0$ near one, $\sup_{z}|\Phi(z/c)-\Phi(z)|\le C|c-1|$. Replacing
		$\sigma_{\nu,n}$ by $\widetilde\sigma_{n}$ therefore adds only $O(n^{-1})$ to the
		bound and proves \cref{eq-mc-03}.
	\end{proof}

\subsection{Proof of \cref{cor:uncorrected}}\label{subsec:corollary-proofs}
	
	The corollary follows by comparing the corrected statistic with its
	deterministic Gaussian shift.
	
	Using $\beta_n=\kappa A_n/n$, we have
		\[
		\frac{\sqrt n(\bar\theta_n-\theta^*)}{\widetilde\sigma_n}
		=
		\frac{\sqrt n(\bar\theta_n-\theta^*-\beta_n)}
		{\widetilde\sigma_n}
		+
		\frac{\kappa A_n}{\widetilde\sigma_n\sqrt n}.
		\]
		Hence, by \cref{eq-mc-03},
		\begin{equation}
			\sup_{z\in\mathbb R}
			\Big|
			\mathbb P\big(
			\frac{\sqrt n(\bar\theta_n-\theta^*)}{\widetilde\sigma_n}
			\le z
			\big)
			-
			\Phi\big(
			z-\frac{\kappa A_n}{\widetilde\sigma_n\sqrt n}
			\big)
			\Big|
			\le Cn^{-\alpha/2}.
			\label{eq-pd-02}
		\end{equation}
		
		Put Put $\mu_n=\frac{\sqrt n\,\beta_n}{\widetilde\sigma_n}$, by \cref{lem:An-asymptotic,lem:oracle-variance}, it follow that
		\begin{equation}
			\mu_n=\frac{\kappa A_n}{\widetilde\sigma_n\sqrt n}
			=
			\frac{\kappa\epsilon_0}
			{\bar a(1-\alpha)\sigma_\infty}
			n^{1/2-\alpha}
			+
			o\bigl(n^{1/2-\alpha}\bigr).
			\label{eq-pd-03}
		\end{equation}
		
		The mean value theorem and
		$\sup_{z\in\mathbb R}\phi(z)=1/\sqrt{2\pi}$ give
		\[
		\sup_{z\in\mathbb R}
		\big|\Phi(z-\mu_n)-\Phi(z)\big|
		\le
		\frac{\abs{\mu_n}}{\sqrt{2\pi}}.
		\]
		Combining this estimate with \cref{eq-pd-02} and \cref{eq-pd-03}, we obtain
		\[
		\sup_{z\in\mathbb R}
		\Bigg|
		\mathbb P\Bigl(
		\frac{\sqrt n(\bar\theta_n-\theta^*)}
		{\widetilde\sigma_n}
		\le z
		\Bigr)
		-\Phi(z)
		\Bigg|
		\le Cn^{1/2-\alpha}.
		\]
		This proves \cref{eq-mc-04}.
		
		If $\kappa=0$, then $\beta_n=0$, and \cref{eq-pd-02} directly gives
		\[
		\sup_{z\in\mathbb R}
		\Big|
		\mathbb P\Big(
		\frac{\sqrt n(\bar\theta_n-\theta^*)}{\widetilde\sigma_n}
		\le z
		\Big)-\Phi(z)
		\Big|
		\le Cn^{-\alpha/2}.
		\]
		
		Now suppose that $\kappa\ne0$. The distance between two unit-variance
		normal distribution functions with a mean shift
		$\kappa A_n/(\widetilde\sigma_n\sqrt n)$ satisfies
		\begin{equation}
			\sup_{z\in\mathbb R}
			\Big|
				\Phi\Big(
				z-\frac{\kappa A_n}{\widetilde\sigma_n\sqrt n}
				\Big)-\Phi(z)\Big|
			=
			2\Phi\Big(
			\frac12
			\big|
			\frac{\kappa A_n}{\widetilde\sigma_n\sqrt n}
			\big|
			\Big)-1.
			\label{eq-pd-04}
		\end{equation}
		When $\mu_n>0$, the function $\Phi(z)-\Phi(z-\mu_n)$ is positive, and its derivative $\phi(z)-\phi(z-\mu_n)$ changes sign only at $z=\mu_n/2$. Hence its maximum is attained at $z=\mu_n/2$. The case $u<0$ follows by symmetry.
		
		Since the shift in \cref{eq-pd-03} converges to zero, Taylor expansion
		of $\Phi$ at zero gives
		\begin{equation}
			\begin{aligned}
				&\sup_{z\in\mathbb R}
				\Big|
				\Phi\Big(
				z-\frac{\kappa A_n}{\widetilde\sigma_n\sqrt n}
				\Big)-\Phi(z)\Big|
				=
				\frac{1}{\sqrt{2\pi}}
				\big|
				\frac{\kappa A_n}{\widetilde\sigma_n\sqrt n}
				\big|
				+
				O\Big(
				\big|
				\frac{\kappa A_n}{\widetilde\sigma_n\sqrt n}
				\big|^3
				\Big).
			\end{aligned}
			\label{eq-pd-05}
		\end{equation}
		
		By the reverse triangle inequality for the supremum norm and
		\cref{eq-pd-02},
		\begin{equation}
			\begin{aligned}
				\Bigg|
				&\sup_{z\in\mathbb R}
				\Big|
				\mathbb P\Big(
				\frac{\sqrt n(\bar\theta_n-\theta^*)}{\widetilde\sigma_n}
				\le z
				\Big)-\Phi(z)
				\Big|
				-
				\sup_{z\in\mathbb R}
				\Big|
					\Phi\Big(
					z-\frac{\kappa A_n}{\widetilde\sigma_n\sqrt n}
					\Big)-\Phi(z)\Big|
				\Bigg|
				\le Cn^{-\alpha/2}.
			\end{aligned}
			\label{eq-pd-06}
		\end{equation}
		
		Finally, \cref{eq-pd-03} implies
		\[
		\big|
		\frac{\kappa A_n}{\widetilde\sigma_n\sqrt n}
		\big|
		=
		\frac{\abs{\kappa}\epsilon_0}
		{\bar a(1-\alpha)\sigma_\infty}
		n^{1/2-\alpha}
		+
		o\bigl(n^{1/2-\alpha}\bigr).
		\]
		Moreover,
		\[
		\big|
		\frac{\kappa A_n}{\widetilde\sigma_n\sqrt n}
		\big|^3
		=o\bigl(n^{1/2-\alpha}\bigr),
		\qquad
		n^{-\alpha/2}
		=o\bigl(n^{1/2-\alpha}\bigr),
		\]
		where the second relation follows from $\alpha<1$. Combining
		\cref{eq-pd-05} and \cref{eq-pd-06} therefore gives
		\[
		\sup_{z\in\mathbb R}
		\Big|
		\mathbb P\Big(
		\frac{\sqrt n(\bar\theta_n-\theta^*)}{\widetilde\sigma_n}
		\le z
		\Big)-\Phi(z)
		\Big|
		=
		\frac{\abs{\kappa}\epsilon_0}
		{\sqrt{2\pi}\,\bar a(1-\alpha)\sigma_\infty}
		n^{1/2-\alpha}
		+
		o\bigl(n^{1/2-\alpha}\bigr).
		\]
		This proves \cref{eq-mc-05}.

	\begin{acks}
		This work was supported by the National Natural Science Foundation of China.
	\end{acks}
	
	\medskip\noindent
	\textbf{Conflict of interest.} The authors declare that they have no conflict of interest.
	
	\bibliographystyle{imsart-nameyear}
    \bibliography{refs.bib}
	
	\clearpage
	\appendix
	\raggedbottom
	
	\section{Proof of Preliminary lemmas}

	\subsection{Proof of \cref{lem:markov-toolkit,lem:An-asymptotic}}\label[appendix]{app:common-tools}
	We will present two auxiliary lemmas. The \cref{lem:geometric-consequence} provides an explanation for the geometric mixing property of Markov data, \cref{lem:PR-weights} contains some bounded conditions that will be utilized in the proof.
	\begin{lemma}\label[lemma]{lem:geometric-consequence}
		Under \cref{ass:density-mixing}, there exist constants $C<\infty$ and $\rho\in(0,1)$ such that
		\begin{equation}
			\abs{P^t(x,A)-\pi(A)}
			\le C\rho^t\pi(A),
			\qquad t\ge1,\quad x\in S,\quad A\in\mathcal S.
			\label{eq-ma-03}
		\end{equation}
		In particular, $\sum_{t\ge1}\rho^t<\infty$.
	\end{lemma}
	
	\begin{proof}
		By \cref{eq-ma-01}, the transition kernel admits a density $p_1$
		with respect to $\pi$ satisfying $0\le p_1(x,y)\le K$. Hence
		\[
		\int_{S\times S}p_1(x,y)^2\,\pi(dx)\pi(dy)\le K^2,
		\]
		so $P$ is a compact operator on $L^2(\pi)$. Since the stationary chain is mixing by positive Harris recurrence and aperiodicity, $1$ is the only peripheral eigenvalue of $P$ and is simple. Therefore, for every centered $f\in L^2(\pi)$,
		\begin{equation}
			\norm{P^m f}_{2}
			\le C\rho^m\norm{f}_{2},
			\qquad m\ge0.
			\label{eq:geometric-operator-bound}
		\end{equation}
		
		For $\pi$-almost every $y$, invariance gives
		\[
		\int_S\bigl[p_1(x,y)-1\bigr]\pi(dx)=0,
		\qquad
		\norm{p_1(\cdot,y)-1}_{2}\le C.
		\]
		Using the Chapman--Kolmogorov identity, Cauchy--Schwarz, and
		\cref{eq:geometric-operator-bound},
		\begin{align*}
			\abs{p_{m+2}(x,y)-1}
			&=
			\abs{
				P\Bigl[P^m\bigl\{p_1(\cdot,y)-1\bigr\}\Bigr](x)
			} \\
			&\le
			C\big\| P^m\bigl\{p_1(\cdot,y)-1\bigr\}\big\|_{2}
			\le
			C\rho^m,
			\qquad m\ge0.
		\end{align*}
		
		Choose $t_0\ge2$ such that $C\rho^{t_0-2}\le1/2$. Then
		\[
		\frac12\le p_{t_0}(x,y)\le\frac32,
		\]
		and integration over $A$ gives
		\begin{equation}
			\frac12\pi(A)\le P^{t_0}(x,A)\le\frac32\pi(A).
			\label{eq-ma-02}
		\end{equation}
		
		Consequently,
		\[
		0\le2P^{t_0}(x,A)-\pi(A)\le2\pi(A),
		\qquad
		\int_S\bigl[2P^{t_0}(x,A)-\pi(A)\bigr]\pi(dx)=\pi(A).
		\]
		Since $P^{t_0}=\frac12\pi+\frac12\bigl(2P^{t_0}-\pi\bigr)$, then
		\[
		\abs{P^{rt_0}(x,A)-\pi(A)}
		\le 2^{-r}\pi(A),
		\qquad r\ge1.
		\]
		Writing $t=rt_0+s$ with $0\le s<t_0$, using
		\cref{eq-ma-01} for the remaining $s$ steps, gives
		\[
		\abs{P^t(x,A)-\pi(A)}
		\le C2^{-r}\pi(A)
		\le C\rho^t\pi(A),
		\qquad t\ge t_0,
		\]
		where $\rho=2^{-1/t_0}$. For $1\le t<t_0$, \cref{eq-ma-01} and
		$\rho^t\ge1/2$ give the same bound after enlarging $C$. Thus
		\begin{equation*}
			\abs{P^t(x,A)-\pi(A)}
			\le C\rho^t\pi(A),
			\qquad t\ge1,
		\end{equation*}
		which completes the proof.
	\end{proof}

	\begin{lemma}\label[lemma]{lem:PR-weights}
		There is a constant $C<\infty$, independent of $n$, such that
		\begin{equation}
			\max_{0\le i\le n}\abs{q_{i,n}}\le C,
			\qquad
			\abs{q_{i+1,n}-q_{i,n}}\le C\epsilon_i,
			\quad 1\le i<n.
			\label{eq-ph-01}
		\end{equation}
		Moreover,
		\begin{equation}
			\frac1n\sum_{i=1}^n
			\Big(q_{i,n}-\frac1{\bar a}\Big)^2
			\longrightarrow0.
			\label{eq-ph-02}
		\end{equation}
		More precisely, let $L_n=\lceil K n^\alpha\log(en)\rceil$, where $K$ is sufficiently large. If $k_0\to\infty$ and $k_0=o(n)$, then
		\begin{equation}
			\sup_{k_0\le i\le n-L_n}
			\abs{q_{i,n}-\frac1{\bar a}}
			\longrightarrow0.
			\label{eq-ph-04}
		\end{equation}
	\end{lemma}
	
	\begin{proof}
		For all sufficiently large $i$, $0<\bar a\epsilon_i\le1/2$. By
		\cref{eq-mb-01},
		\[
		\abs{q_{i,n}}
		\le
		\epsilon_i\sum_{\ell=0}^{\infty}
		\exp\Big(
		-\bar a\sum_{s=1}^{\ell}\epsilon_{i+s}
		\Big).
		\]
		For $\ell\le i$ and $\ell>i$, respectively,
		\[
		\sum_{s=1}^{\ell}\epsilon_{i+s}
		\ge 2^{-\alpha}\ell\epsilon_i,
		\qquad
		\sum_{s=1}^{\ell}\epsilon_{i+s}
		\ge
		C^{-1}\bigl\{(i+\ell)^{1-\alpha}-i^{1-\alpha}\bigr\}.
		\]
		Integral comparison therefore gives
		\[
		\sum_{\ell=0}^{i}
		\exp\Big(
		-\bar a\sum_{s=1}^{\ell}\epsilon_{i+s}
		\Big)
		+
		\sum_{\ell>i}
		\exp\Big(
		-C^{-1}\bigl\{(i+\ell)^{1-\alpha}-i^{1-\alpha}\bigr\}
		\Big)
		\le C\epsilon_i^{-1}.
		\]
		Hence $\abs{q_{i,n}}\le C$; the finitely many remaining indices are
		controlled directly by \cref{eq-mb-01}. The backward relation
		\[
		q_{i,n}
		=
		\epsilon_i+
		\epsilon_i\epsilon_{i+1}^{-1}
		(1-\bar a\epsilon_{i+1})q_{i+1,n}
		\]
		and $\epsilon_i/\epsilon_{i+1}-1\le C\epsilon_i$ then give
		\[
		\abs{q_{i+1,n}-q_{i,n}}\le C\epsilon_i.
		\]
		This proves \cref{eq-ph-01}.
		
		For $\ell\epsilon_i$ in a bounded range,
		\[
		\sum_{s=1}^{\ell}\epsilon_{i+s}
		=\ell\epsilon_i\{1+o(1)\},
		\qquad
		\sum_{s=1}^{\ell}\epsilon_{i+s}^{2}=o(1),
		\]
		therefore,
		\[
		\epsilon_i\sum_{\ell=0}^{\infty}
		\prod_{s=1}^{\ell}(1-\bar a\epsilon_{i+s})
		\longrightarrow
		\int_0^\infty e^{-\bar a u}\,du
		=\frac1{\bar a}.
		\]
		The tail omitted from the finite sum in \cref{eq-mb-01} satisfies
		\[
		Cn^\alpha
		\exp\bigl\{-C^{-1}(n-i)n^{-\alpha}\bigr\}
		\le
		Cn^\alpha(en)^{-K/C}
		=o(1),
		\qquad i\le n-L_n,
		\]
		when $K$ is sufficiently large. Thus \cref{eq-ph-04} holds.
		
		Finally, \cref{eq-ph-01} and \cref{eq-ph-04} give
		\[
		\frac1n\sum_{i=1}^n
		\Big(q_{i,n}-\frac1{\bar a}\Big)^2
		\le
		\frac{Ck_0}{n}
		+
		\sup_{k_0\le i\le n-L_n}
		\Big(q_{i,n}-\frac1{\bar a}\Big)^2
		+
		\frac{CL_n}{n}.
		\]
		Since $k_0=o(n)$ and $L_n=o(n)$, the right-hand side converges to zero,
		which proves \cref{eq-ph-02}.
	\end{proof}
	
	Now, we will prove \cref{lem:markov-toolkit}.
	\begin{proof}[Proof of \cref{lem:markov-toolkit}]
		Let $\psi$ be measurable with respect to
		$\sigma(X_{r+t},X_{r+t+1},\ldots)$. By the Markov property,
		\[
		\E[\psi\mid\mathcal F_r]
		=
		\int_S\E_x\psi\,P^t(X_r,dx),
		\qquad
		\E\psi
		=
		\int_S\E_x\psi\,(\nu P^{r+t})(dx).
		\]
		The density mixing assumption implies
		\begin{align*}
			\big|\E[\psi\mid\mathcal F_r]
				-\int_S\E_x\psi\,\pi(dx)\big|
			&\le
			C\rho^t\int_S\E_x\abs{\psi}\,\pi(dx),\\
			\big|\E\psi-\int_S\E_x\psi\,\pi(dx)\big|
			&\le
			C\rho^{r+t}\int_S\E_x\abs{\psi}\,\pi(dx).
		\end{align*}
		Therefore, since $r\ge0$,
		\[
		\abs{\E[\psi\mid\mathcal F_r]-\E\psi}
		\le
		C\rho^t\int_S\E_x\abs{\psi}\,\pi(dx).
		\]
		Consequently,
		\[
		\abs{\cov(\varphi,\psi)}
		=
		\Big|\E\bigl[
			\varphi\{\E[\psi\mid\mathcal F_r]-\E\psi\}
			\bigr]\Big|
		\le
		C\rho^t\E\abs{\varphi}
		\int_S\E_x\abs{\psi}\,\pi(dx),
		\]
		which proves \cref{eq-pe-01}.
		
		For the second assertion, the density mixing assumption gives, for
		$i\ge1$,
		\[
		\E f(X_i)^2\le C\int_S f(x)^2\pi(dx),
		\qquad
		\abs{\E f(X_i)}
		\le
		C\rho^i\int_S\abs{f(x)}\pi(dx).
		\]
		For $j>i$, apply part~\textup{(i)} with
		$\varphi=f(X_i)$ and $\psi=f(X_j)$. Then $\abs{\cov\bigl(f(X_i),f(X_j)\bigr)}
		\le C_f\rho^{j-i}$.
		
		Hence, using the covariance expansion,
		\begin{align*}
			\E\bigg(
			\sum_{i\ge1}c_i\bigl\{f(X_i)-\E f(X_i)\bigr\}
			\bigg)^2
			&\le
			C_f\sum_{i\ge1}c_i^2
			+
			2C_f\sum_{1\le i<j}\abs{c_ic_j}\rho^{j-i} \\
			&\le
			C_f\sum_{i\ge1}c_i^2,
		\end{align*}
		where we used the summability of $(\rho^j)_{j\ge1}$ and
		$\abs{c_ic_j}\le(c_i^2+c_j^2)/2$. Moreover,
		\[
		\Big|\E\sum_{i\ge1}c_i f(X_i)\Big|
		\le
		C_f\sum_{i\ge1}\abs{c_i}\rho^i.
		\]
		Combining the centered second-moment bound with the squared-mean bound
		proves \cref{eq-pe-02}.
	\end{proof}
	
	Next, we will prove \cref{lem:oracle-variance}.
	\begin{proof}[Proof of \cref{lem:oracle-variance}]
		Since $\delta$ is centered, the density estimate in
		\cref{eq-ma-03}, integrated against $\delta$, gives
		\[
		\abs{P^t\delta(x)}
		\le
		C\rho^t\int_S\abs{\delta(y)}\pi(dy),
		\qquad t\ge1.
		\]
		Consequently,
		\[
		\Big|\int_S\delta(x)P^t\delta(x)\,\pi(dx)\Big|
		\le C\rho^t,
		\qquad t\ge1.
		\]
		Since $\delta\in L^2(\pi)$, the terms on the right are summable. Using
		the definition of $\widetilde\sigma_n^2$ in \cref{eq-mb-05} and grouping
		the off-diagonal terms according to the lag gives
		\begin{equation}
			\widetilde\sigma_n^2
			=
			\frac1n\sum_{i=1}^n q_i^2
			\int_S\delta(x)^2\pi(dx)
			+
			\frac2n\sum_{t=1}^{n-1}
			\bigg(
			\int_S\delta(x)P^t\delta(x)\,\pi(dx)
			\bigg)
			\sum_{i=1}^{n-t}q_iq_{i+t}.
			\label{eq-pi-03}
		\end{equation}
		
		For every fixed $t\ge1$, \cref{eq-ph-01,eq-ph-02} imply
		\[
		\frac1n\sum_{i=1}^{n-t}q_iq_{i+t}
		\longrightarrow
		\frac1{\bar a^2}.
		\]
		Indeed, the uniform boundedness of $q_i$ in \cref{eq-ph-01} gives
		\[
		\abs{q_iq_{i+t}-\frac1{\bar a^2}}
		\le
		C\bigg(
		\abs{q_i-\frac1{\bar a}}
		+
		\abs{q_{i+t}-\frac1{\bar a}}
		\bigg),
		\]
		and the average of the right-hand side converges to zero by
		\cref{eq-ph-02} and the Cauchy--Schwarz inequality. Moreover,
		\[
		\big|\frac1n\sum_{i=1}^{n-t}q_iq_{i+t}\big|\le C,
		\qquad 1\le t<n.
		\]
		Thus the summands in the second term of \cref{eq-pi-03} are dominated by
		the summable sequence $C\rho^t$. Applying dominated convergence and using
		\cref{eq-mb-04} yields
		\[
		\widetilde\sigma_n^2
		\longrightarrow
		\frac1{\bar a^2}
		\bigg\{
		\int_S\delta(x)^2\pi(dx)
		+
		2\sum_{t=1}^{\infty}
		\int_S\delta(x)P^t\delta(x)\,\pi(dx)
		\bigg\}.
		\]
		This proves \cref{eq-pi-01}.
		
		Next, applying \cref{eq-pe-01} with $\varphi=1$, $r=0$, $t=i$, and
		$\psi=\delta(X_i)$ gives
		\[
		\Big|\int_S\delta(x)(\nu P^i)(dx)\Big|
		\le C\rho^i.
		\]
		Therefore, by the boundedness of $q_i$ in \cref{eq-ph-01} and the
		definition of $\mu_{\nu,n}$,
		\[
		\abs{\mu_{\nu,n}}
		\le
		\frac{C}{n}\sum_{i=1}^n\rho^i
		\le
		\frac{C}{n}
		\le C.
		\]
		
		The density estimate in \cref{eq-ma-03}, integrated first against $\nu$
		and then against $\delta^2$, gives
		\[
		\Big|
			\int_S\delta(x)^2(\nu P^i)(dx)
			-
			\int_S\delta(x)^2\pi(dx)\Big|
		\le C\rho^i.
		\]
		For $j>i$, the same estimate and the bound on $P^{j-i}\delta$ give
		\[
		\begin{aligned}
			&\Big|
				\int_S\delta(x)P^{j-i}\delta(x)(\nu P^i)(dx)
				-
				\int_S\delta(x)P^{j-i}\delta(x)\pi(dx)
			\Big| 
			\le
			C\rho^i
			\int_S\abs{\delta(x)P^{j-i}\delta(x)}\pi(dx)
			\le C\rho^j.
		\end{aligned}
		\]
		Using the definitions of $\sigma_{\nu,n}^2$ and
		$\widetilde\sigma_n^2$, together with \cref{eq-ph-01}, we obtain
		\begin{align*}
			n\abs{\sigma_{\nu,n}^2-\widetilde\sigma_n^2}
			&\le
			C\sum_{i=1}^n\rho^i
			+
			C\sum_{1\le i<j\le n}\rho^j
			+
			n^2\mu_{\nu,n}^2 \\
			&\le
			C\sum_{i\ge1}\rho^i
			+
			C\sum_{j\ge2}j\rho^j
			+
			C
			\le C.
		\end{align*}
		Thus $\abs{\sigma_{\nu,n}^2-\widetilde\sigma_n^2}\le Cn^{-1}$,
		which proves \cref{eq-pi-04}. Since $\widetilde\sigma_n^2\to\Gamma/\bar a^2>0$ and
		\cref{eq-pi-04} holds uniformly in $\nu$, there exists an integer
		$n_0\ge1$, independent of $\nu$, such that $\sigma_{\nu,n}^2\ge\frac{\Gamma}{4\bar a^2}$, $n\ge n_0$.
	\end{proof}
	
	Following the above, we will then proceed to prove \cref{lem:An-asymptotic}.
	\begin{proof}[Proof of \cref{lem:An-asymptotic}]
		Take $k_0=\lceil\log(en)\rceil$ and $L_n$ as in
		\cref{lem:PR-weights}. By \cref{eq-mb-07,eq-ph-01,eq-ph-04},
		for all sufficiently large $n$,
		\begin{equation}
		    \begin{aligned}
			\Big|A_n-\frac1{\bar a}\sum_{i=1}^{n-1}\epsilon_i\Big|
			&=
			\Big|\sum_{i=1}^{n-1}
				\epsilon_i\big(q_i-\frac1{\bar a}\big)\Big|\\
			&\le
			C\sum_{i<k_0}\epsilon_i
			+\sup_{k_0\le i\le n-L_n}
			\big|q_i-\frac1{\bar a}\big|
			\sum_{i=1}^{n-1}\epsilon_i
			+C L_n\epsilon_{n-L_n}.	
		\end{aligned}\label{eq-ph-15}
		\end{equation}
		Indeed, the first and last terms bound the initial and terminal
		index ranges, while \cref{eq-ph-04} controls the middle range.
		
		By \cref{eq-in-01} and the expression for $L_n$ in
		\cref{lem:PR-weights}, $L_n=o(n)$ and
		\[
		\sum_{i<k_0}\epsilon_i
		=O\bigl((\log(en))^{1-\alpha}\bigr),
		\qquad
		L_n\epsilon_{n-L_n}=O(\log(en)).
		\]
		Moreover, \cref{eq-ph-04} gives
		\[
		\sup_{k_0\le i\le n-L_n}
		\abs{q_i-\frac1{\bar a}}=o(1),
		\]
		and integral comparison yields
		\[
		\sum_{i=1}^{n-1}\epsilon_i
		=\frac{\epsilon_0}{1-\alpha}n^{1-\alpha}+O(1).
		\]
		Since $\log(en)=o(n^{1-\alpha})$, \cref{eq-ph-15} implies
		\[
		A_n-\frac1{\bar a}\sum_{i=1}^{n-1}\epsilon_i
		=o(n^{1-\alpha}).
		\]
		Consequently, by \cref{eq-mb-07},
		\[
		A_n\sim
		\frac{\epsilon_0}{\bar a(1-\alpha)}n^{1-\alpha},
		\qquad
		\beta_n=\frac{\kappa A_n}{n}
		=\frac{\kappa\epsilon_0}{\bar a(1-\alpha)}n^{-\alpha}
		+o(n^{-\alpha}).
		\]
	\end{proof}
	
	\subsection{Proofs of \cref{lem:bridge-kernel,lem:exchangeable-construction,lem:projection-sum}}\label[appendix]{app:auxiliary-lemmas}
	In this section, we will prove \cref{lem:bridge-kernel,lem:exchangeable-construction,lem:projection-sum}.
	\begin{proof}[Proof of \cref{lem:bridge-kernel}]
		For the compatible versions of $p_t$ fixed before
		\cref{lem:bridge-kernel}, the Chapman--Kolmogorov identities hold
		pointwise and
		\[
		\int_Sp_t(z,y)\pi(dz)=1,
		\qquad
		\abs{p_t(x,y)-1}\le C\rho^t.
		\]
		Choose $m_0$ so that $C\rho^{m_0}\le1/2$. Then
		\[
		p_t(x,y)\ge\frac12,
		\qquad t\ge m_0.
		\]
		The marginal densities
		$d_j(y)=\int_Sp_j(x,y)\nu(dx)$ therefore satisfy the bounds stated
		before \cref{lem:bridge-kernel}.
		
		The Markov property and disintegration give the conditional densities
		in part~\ref{item:bridge-density}. In the two-sided case, conditional on
		$X_{\ell-1}=x$ and $X_{r+1}=y$, the density of $X_j$ is
		\[
		z\longmapsto
		\frac{
			p_{j-\ell+1}(x,z)p_{r-j+1}(z,y)
		}{
			p_{r-\ell+2}(x,y)
		}.
		\]
		If only the right boundary is present, the density is
		\[
		z\longmapsto
		\frac{d_j(z)p_{r-j+1}(z,y)}{d_{r+1}(y)};
		\]
		if only the left boundary is present or neither boundary is present, it is,
		respectively,
		\[
		z\longmapsto p_{j-\ell+1}(x,z),
		\qquad
		z\longmapsto d_j(z).
		\]
		The denominator in the two-sided case is at least $1/2$ when
		$r-\ell+2\ge2m_0$, and the denominator in the right-boundary-only case
		is at least $1/2$ when $r+1\ge m_0$. The exceptional boundary values form
		a null set and may be assigned any fixed conditional law. This proves
		part~\ref{item:bridge-density}.
		
		Moreover, $0\le p_t(x,y)\le C$ and $0\le d_j(y)\le C$. Together with
		the preceding denominator bounds, these estimates give the uniform bounds
		in part~\ref{item:bridge-uniform-bound}.
		
		For part~\ref{item:bridge-forgetting}, if both boundaries are present and
		$j-\ell+1,r-j+1\ge m$, then
		\[
		\sup_z\bigg|
		\frac{
			p_{j-\ell+1}(x,z)p_{r-j+1}(z,y)
		}{
			p_{r-\ell+2}(x,y)
		}
		-1
		\bigg|
		\le C\rho^m.
		\]
		If only the left boundary is present,
		\[
		\sup_z\big|p_{j-\ell+1}(x,z)-1\big|\le C\rho^m.
		\]
		If only the right boundary is present,
		\[
		\sup_z\bigg|
		\frac{d_j(z)p_{r-j+1}(z,y)}{d_{r+1}(y)}-1
		\bigg|
		\le
		C\bigl(\rho^j+\rho^{r+1-j}\bigr).
		\]
		Finally,
		\[
		\sup_z\abs{d_j(z)-1}\le C\rho^j.
		\]
		Integrating these bounds against the centered function $f$ gives
		\[
		\big|\E\{f(X_j)\mid\mathcal F_{\ell,r}\}\big|
		\le
		C\bigl(\rho^m+\one_{\{\ell=1\}}\rho^j\bigr)
		\int_S\abs{f(x)}\pi(dx),
		\]
		and, when neither boundary is present,
		\[
		\abs{\E f(X_j)}
		\le
		C\rho^j\int_S\abs{f(x)}\pi(dx).
		\]
		This proves part~\ref{item:bridge-forgetting}.
		
		The conditional laws obtained above are measurable probability kernels in
		the boundary values. Since the finite product spaces are standard Borel,
		the randomization lemma gives a measurable realization using one 
        independent $\operatorname{Unif}(0,1)$ random variable. When $\ell=1$,
		the same argument applied to the conditional law of
		$(X_0,\ldots,X_r)$ gives the required realization including $X_0$.
		Thus part~\ref{item:bridge-realization} follows.
	\end{proof}
	
	\begin{proof}[Proof of \cref{lem:exchangeable-construction}]
		Apply part~\ref{item:bridge-realization} of \cref{lem:bridge-kernel} to
		each $\mathcal B_i$ with seed $I_i$. Conditional on $\mathcal F$, the
		seeds are independent and each resampled block has the same conditional
		law as the original block. The construction is measurable with respect to
		$\sigma(X_{\mathcal L_i},I_i)$, and averaging the conditional identity gives
		\[
		(X_0^{(i)},X_1^{(i)},\ldots,X_n^{(i)})
		\stackrel{d}{=}(X_0,X_1,\ldots,X_n).
		\]
		If $r_i<n$, let $s=r_i+1$. The bridge segment and $I_i$ are measurable
		with respect to the past together with $X_s$; the Markov property therefore
		makes the later future conditionally independent given $X_s$. Thus
		\[
		\E\{f(X_j)\mid\mathcal F_{j-1}\vee\sigma(I_i)\}=Pf(X_{j-1}),
		\qquad f\in L^1(\pi),\quad j\ge s+1.
		\]
		These two identities prove both assertions.
	\end{proof}
	
	\begin{proof}[Proof of \cref{lem:projection-sum}]
		Set $S_{Z}=\sum_{i=1}^n\{Z_{i}-\E Z_{i}\}$.
		By linearity and the $L^2$ contraction property of conditional expectation,
		\[
		\Big\|
		\sum_{i=1}^n\{\E(Z_{i}\mid\mathcal G)-\E Z_{i}\}
		\Big\|_{2}
		=\|\E(S_{Z}\mid\mathcal G)\|_{2}
		\le\|S_{Z}\|_{2}.
		\]
		The covariance assumption gives
		\[
		\|S_{Z}\|_{2}^2
		\le
		\sum_{i=1}^n\sum_{j=1}^n\abs{\cov(Z_{i},Z_{j})}.
		\]
		The pairs with $\abs{i-j}\le8m+4$ contribute at most $Cnm^5$, while
		\[
		\sum_{\abs{i-j}>8m+4}m^4\rho^{\lfloor\abs{i-j}/2\rfloor}
		\le Cnm^4\le Cnm^5.
		\]
		Thus $\|S_{Z}\|_{2}^2\le Cnm^5$. The last assertion follows when
		$\E Z_{i}=0$ for all $i$.
	\end{proof}
		
	\section{Proof of \cref{lem:corrected-remainder}}\label[appendix]{app:corrected-remainder}
	This appendix will prove \cref{lem:corrected-remainder}, as well as estimate the bias-corrected remainder term $U_c$ and its variation under local path replacement $U_c - U_c'$. Since $h(x)=a(x)-\bar a$, the first estimate in \cref{eq-qc-02} follows from
	\begin{equation*}
		\begin{aligned}
		\E\abs{U_c}
		&=
		\frac{1}{\sigma_{\nu,n}\sqrt n}
		\E\abs{
			\frac{q_0\Delta_0}{\epsilon_0}
			-\sum_{i=1}^n q_i h(X_i)\Delta_{i-1}
			-\kappa A_n-\mu_{\nu,n}}\\
		&\le
		\frac{1}{\sigma_{\nu,n}\sqrt n}
		\bigg\{
		\E\abs{\sum_{i=1}^nq_i\Delta_{i-1}\xi_i}
		+C\sum_{i=1}^{n-1}\epsilon_i^{3/2}
		+\bigg\{
		\E\bigg[
		\sum_{i=1}^{n-1}q_i\epsilon_i
		\big\{\delta(X_i)v_h(X_i)-\kappa\big\}
		\bigg]^2
		\bigg\}^{1/2}
		+C
		\bigg\}.
	\end{aligned}
	\end{equation*}
	
	For the second estimate, the deterministic terms in $U_c-U_c'$ cancel.
	Conditioning on the uniform index $I$ gives
	\begin{align*}
		\E\abs{U_c-U_c'}^2
		&=
		\frac{1}{n^2\sigma_{\nu,n}^2}
		\sum_{i=1}^{n}
		\E\bigg[
		\sum_{j=1}^{n}q_j
		\big\{-h(X_j)\Delta_{j-1}
		+h(X_j^{(i)})\Delta_{j-1}^{(i)}\big\}
		\bigg]^2.
	\end{align*}
	
	To estimate the above items, several auxiliary lemmas are presented below to establish the required moment bounds and local path replacement estimates.
	
	For a centered $f\in L^p(\pi)$, $p\ge1$, set $v_f(x)=\sum_{t=1}^{\infty}P^tf(x)$ and $g_f=f+v_f$. By \cref{eq-ma-03},
	$\abs{P^tf(x)}\le C\rho^t\int_S\abs{f(y)}\pi(dy)$, so the series for $v_f$ converges absolutely and uniformly and $v_f$ is bounded. Consequently $g_f\in L^p(\pi)$, and termwise application of $P$ gives
	\begin{equation}
		g_f-Pg_f=f,
		\qquad Pg_f=v_f.
		\label{eq-pg-01}
	\end{equation}
	Moreover,
	\begin{align}
		\norm{v_f}_{\infty}
		&\le C\bigg(\sum_{t=1}^{\infty}\rho^t\bigg)
		\int_S\abs{f(x)}\pi(dx),
		\label{eq-pg-02}\\
		\sup_{x\in S}P\abs{g_f}^p(x)
		&\le C_p\int_S\abs{f(x)}^p\pi(dx).
		\label{eq-pg-03}
	\end{align}
	The second estimate follows from $p_1(x,y)\le1+C\rho$, the inequality $\abs{f+v_f}^p\le2^{p-1}(\abs f^p+\abs{v_f}^p)$, and Jensen's inequality.

    First, establish the second-order and fourth-order moment bounds of the recursive error, which will be used for the subsequent estimation of weighted sums and local path replacements.
	\begin{lemma}\label[lemma]{lem:delta-moments}
		Under \cref{ass:density-mixing,ass:fourth-moments}, for the recursion in
		\cref{eq-ka-02}, there exists $C<\infty$ such that
		\begin{equation}
			\E\Delta_k^2\le C\epsilon_k,
			\qquad
			\E\Delta_k^4\le C\epsilon_k^2,
			\qquad k\ge1.
			\label{eq-qb-01}
		\end{equation}
	\end{lemma}

	\begin{proof}
		By \cref{eq-pg-01,eq-pg-02}, $g_h,g_\delta\in L^4(\pi)$ and
		$v_h=Pg_h$, $v_\delta=Pg_\delta$ are bounded. Also,
		\begin{equation}
			\abs{\epsilon_{k-1}-\epsilon_k}
			\le Ck^{-\alpha-1}\le C\epsilon_k^2.
			\label{eq-qa-08}
		\end{equation}
		More generally, for $p\in\{1,2\}$,
		\[
		\epsilon_{k-1}^{p}-\epsilon_k^{p}
		=O(k^{-\alpha p-1})
		=o\big(\epsilon_k\epsilon_{k-1}^{p}\big).
		\]
		The recursion \cref{eq-ka-02} and the deterministic initial value
		$\Delta_0$ allow the following Poisson decomposition.
		For $r\in\{h,\delta,g_{h},g_{\delta}\}$ and $q\in\{2,4\}$, domination gives
		\begin{equation}
			P|r|^q(x)\le K\int_S|r(y)|^q\pi(dy),
			\qquad x\in S.
			\label{eq-app-poisson-moment}
		\end{equation}
		Within this proof, define
		\begin{equation}
		\begin{aligned}
			M_{k}&=\{g_{h}(X_{k})-v_{h}(X_{k-1})\}\Delta_{k-1}
			+g_{\delta}(X_{k})-v_{\delta}(X_{k-1}),\\
			V(x,z)&=v_{h}(x)z+v_{\delta}(x).
		\end{aligned}
		\label{eq-app-poisson-martingale}
		\end{equation}
		Since $\Delta_{k-1}$ is $\mathcal F_{k-1}$-measurable and
		$v_{h}=Pg_{h}$, $v_{\delta}=Pg_{\delta}$,
		\begin{equation}
			\E(M_{k}\mid\mathcal F_{k-1})=0,
			\qquad
			\E\bigl(|M_{k}|^q\mid\mathcal F_{k-1}\bigr)
			\le C_{q}\{1+|\Delta_{k-1}|^q\},
			\quad q\in\{2,4\}.
			\label{eq-app-poisson-martingale-bound}
		\end{equation}
		The Poisson equations imply
		\begin{equation}
			h(X_{k})\Delta_{k-1}+\delta(X_{k})
			=M_{k}+V(X_{k-1},\Delta_{k-1})
			-V(X_{k},\Delta_{k-1}).
			\label{eq-app-poisson-decomposition}
		\end{equation}

		Set, for $k\ge1$,
		\[
			\widetilde\Delta_{k}
			=\Delta_{k}-\epsilon_{k}V(X_{k},\Delta_{k-1}).
		\]
		For $k\ge2$, substitution of \cref{eq-app-poisson-decomposition} into the
		original recursion gives
		\begin{equation}
			\widetilde\Delta_{k}
			=(1-\bar a\epsilon_{k})\widetilde\Delta_{k-1}-\epsilon_{k}M_{k}+R_{k},
			\label{eq-app-poisson-recursion}
		\end{equation}
		where
		\begin{align}
			R_{k}={}&\{(1-\bar a\epsilon_{k})\epsilon_{k-1}-\epsilon_{k}\}
			V(X_{k-1},\Delta_{k-2})\notag\\
			&-\epsilon_{k}v_{h}(X_{k-1})
			(\Delta_{k-1}-\Delta_{k-2}).
			\label{eq-app-poisson-remainder}
		\end{align}
		The stepsize bounds above imply $\epsilon_{k-1}/\epsilon_{k} \to 1$, then
		\[
			\big|(1-\bar a\epsilon_{k})\epsilon_{k-1}-\epsilon_{k}\big|
			\le C\epsilon_{k}^2.
		\]
		Moreover, the original recursion and \cref{eq-app-poisson-moment} give, for
		$q\in\{2,4\}$,
		\begin{equation}
			\|\Delta_{k-1}-\Delta_{k-2}\|_{q}
			\le C_{q}\epsilon_{k-1}\{1+\|\Delta_{k-2}\|_{q}\}.
			\label{eq-app-poisson-increment}
		\end{equation}
		Consequently, consecutive stepsizes are comparable and
		\begin{align}
			\|R_{k}\|_{q}&\le C_{q}\epsilon_{k}^2
			\{1+\|\Delta_{k-2}\|_{q}\},\notag\\
			\|\Delta_{k}-\widetilde\Delta_{k}\|_{q}&\le C_{q}\epsilon_{k}
			\{1+\|\Delta_{k-1}\|_{q}\},
			\qquad q\in\{2,4\}.
			\label{eq-app-poisson-transfer}
		\end{align}
		By \cref{eq-app-poisson-moment} and the assumption on $\Delta_{0}$,
		$\Delta_{1}$ and $\widetilde\Delta_{1}$ have finite fourth moments.
		
		For the second moment, put
		$B_{k}=(1-\bar a\epsilon_{k})\widetilde\Delta_{k-1}-\epsilon_{k}M_{k}$.
		By \cref{eq-app-poisson-martingale-bound}, for all sufficiently large $k$,
		\begin{equation}
			\E B_{k}^2
			\le(1-c_{1}\epsilon_{k})\E\widetilde\Delta_{k-1}^2
			+C\epsilon_{k}^2\{1+\E\Delta_{k-1}^2\},
			\label{eq-app-second-recursion}
		\end{equation}
		with some $c_{1}>0$. For fixed $\eta>0$,
		$|u+v|^2\le(1+\eta\epsilon_{k})u^2+C_{\eta}\epsilon_{k}^{-1}v^2$;
		therefore \cref{eq-app-poisson-transfer} yields
		\begin{equation}
			\E\widetilde\Delta_{k}^2
			\le(1-c_{2}\epsilon_{k})\E\widetilde\Delta_{k-1}^2
			+C\epsilon_{k}^2\{1+\E\Delta_{k-1}^2\}
			+C\epsilon_{k}^3\{1+\E\Delta_{k-2}^2\},
			\label{eq-app-second-recursion-total}
		\end{equation}
		for some $c_{2}>0$, after reducing $c_{1}$ if necessary.
		Also, \cref{eq-app-poisson-transfer} and
		$|u+v|^2\le(1+\eta)|u|^2+C_{\eta}|v|^2$ imply
		\begin{equation}
			\E\Delta_{k}^2
			\le(1+\eta)\E\widetilde\Delta_{k}^2
			+C_{\eta}\epsilon_{k}^2\{1+\E\Delta_{k-1}^2\}.
			\label{eq-app-second-transfer}
		\end{equation}
		Choose $A_{2}>1+\eta$. Since
		$\epsilon_{k-1}-\epsilon_{k}=o(\epsilon_{k}\epsilon_{k-1})$, the two inequalities above
		close the induction
		\[
			\E\widetilde\Delta_{k}^2\le K_{2}\epsilon_{k},
			\qquad
			\E\Delta_{k}^2\le A_{2}K_{2}\epsilon_{k}.
		\]
		Indeed, choose $K_{2}$ large enough to absorb the fixed $C\epsilon_{k}^2$ term in
		\cref{eq-app-second-recursion-total}, and then take $k$ large enough to
		absorb the terms proportional to $\epsilon_{k}^2\epsilon_{k-1}$ and
		$\epsilon_{k}^3\epsilon_{k-2}$. The transfer inequality closes the second target
		because $A_{2}>1+\eta$. Enlarging the constant handles the finitely many initial
		indices, so
		\begin{equation}
			\E\widetilde\Delta_{k}^2+\E\Delta_{k}^2\le C\epsilon_{k}.
			\label{eq-app-second-bound}
		\end{equation}

		For the fourth moment, conditional expansion gives
		\begin{align*}
			\E(B_{k}^4\mid\mathcal F_{k-1})
			={}&(1-\bar a\epsilon_{k})^4\widetilde\Delta_{k-1}^4
			+6(1-\bar a\epsilon_{k})^2\epsilon_{k}^2\widetilde\Delta_{k-1}^2
			\E(M_{k}^2\mid\mathcal F_{k-1})\\
			&-4(1-\bar a\epsilon_{k})\epsilon_{k}^3\widetilde\Delta_{k-1}
			\E(M_{k}^3\mid\mathcal F_{k-1})
			+\epsilon_{k}^4\E(M_{k}^4\mid\mathcal F_{k-1}).
		\end{align*}
		The term linear in $M_{k}$ is zero. By \cref{eq-app-poisson-martingale-bound},
		$\E(|M_{k}|^3\mid\mathcal F_{k-1})\le C(1+|\Delta_{k-1}|^3)$.
		Using \cref{eq-app-second-bound}, Cauchy--Schwarz, and H\"older's inequality,
		if
		\[
			\E\widetilde\Delta_{j}^4\le K_{4}\epsilon_{j}^2,
			\qquad
			\E\Delta_{j}^4\le A_{4}K_{4}\epsilon_{j}^2,
		\]
		then the preceding expansion gives
		\begin{equation}
			\E B_{k}^4
			\le(1-2\bar a\epsilon_{k})\E\widetilde\Delta_{k-1}^4
			+C\{\epsilon_{k}^3+(1+A_{4})K_{4}\epsilon_{k}^4\}.
			\label{eq-app-fourth-recursion}
		\end{equation}
		The $\epsilon_{k}^3$ term uses the already established second-moment bound,
		while the remaining terms are $O(K_{4}\epsilon_{k}^4)$ by comparability of
		consecutive stepsizes. Applying
		$|u+v|^4\le(1+\eta\epsilon_{k})|u|^4+C_{\eta}\epsilon_{k}^{-3}|v|^4$ to
		$\widetilde\Delta_{k}=B_{k}+R_{k}$ and using \cref{eq-app-poisson-transfer} gives
		\begin{equation}
			\E\widetilde\Delta_{k}^4
			\le(1-c_{3}\epsilon_{k})\E\widetilde\Delta_{k-1}^4
			+C\{\epsilon_{k}^3+(1+A_{4})K_{4}\epsilon_{k}^4\}.
			\label{eq-app-fourth-recursion-total}
		\end{equation}
		
		Also, $\epsilon_{k-1}^2-\epsilon_{k}^2
		=o(\epsilon_{k}\epsilon_{k-1}^2)$.
		Choose $K_{4}$ large enough to absorb the $C\epsilon_{k}^3$ term into the
		contraction term, and then take $k$ large enough to absorb the
		$O(K_{4}\epsilon_{k}^4)$ term.
		The transfer inequality and
		$|u+v|^4\le(1+\eta)|u|^4+C_{\eta}|v|^4$ give
		\[
			\E|\Delta_{k}|^4
			\le(1+\eta)\E|\widetilde\Delta_{k}|^4
			+C_{\eta}\epsilon_{k}^4\{1+\E|\Delta_{k-1}|^4\}.
		\]
		Taking $A_{4}>1+\eta$ closes the fourth-moment induction. Enlarging the
		constant over the finitely many initial indices proves
		\[
			\E|\widetilde\Delta_{k}|^4+\E|\Delta_{k}|^4\le C\epsilon_{k}^2.
		\]
		The asserted bounds for $\Delta_{k}$ follow.
	\end{proof}

	Apply the Poisson construction above with $p=4$ to $h$ and $\delta$, and write
	\begin{equation*}
		g_{h}-Pg_{h}=h,\qquad g_{\delta}-Pg_{\delta}=\delta,
		\qquad v_{h}=Pg_{h},\qquad v_{\delta}=Pg_{\delta}.
	\end{equation*}
	Then $g_{h},g_{\delta}\in L^4(\pi)$ and $v_{h},v_{\delta}$ are bounded. In particular,
	\begin{equation}
		\begin{aligned}
			\xi_{k}
			&=g_{h}(X_{k})-v_{h}(X_{k-1}),
			\qquad \E(\xi_{k}\mid\mathcal F_{k-1})=0,\\
			h(X_{k})
			&=\xi_{k}+v_{h}(X_{k-1})-v_{h}(X_{k}).
		\end{aligned}
		\label{eq-qa-01a}
	\end{equation}
	
	To control the influence of local substitution on subsequent iterations, the second-order moment decay of homogeneous error recursion and the corresponding summation estimation are presented below.
	\begin{lemma}\label[lemma]{lem:homogeneous-propagation}
		Let $1\le s<t$ be integers. Suppose $H_{s}\in L^2$ is measurable with
		respect to $\mathcal F_{s}\vee\sigma(U)$, where $U$ is a random seed such
		that $\mathcal F_{s}\vee\sigma(U)$ and
		$\sigma(X_{s+1},X_{s+2},\ldots)$ are
		conditionally independent given $X_{s}$. If
		\begin{equation}
			H_{k}=\big(1-\epsilon_{k}a(X_{k})\big)H_{k-1},
			\qquad
			s<k\le t,
			\label{eq-qd-01}
		\end{equation}
		then there are constants $c,C>0$, depending only on the fixed model
		parameters, such that
		\begin{equation}
			\E H_{t}^2
			\le C\exp\big(-c\sum_{k=s+1}^t\epsilon_{k}\big)\E H_{s}^2.
			\label{eq-qd-02}
		\end{equation}
		Moreover,
		\begin{align}
			\sum_{t=s}^{\infty}
			\exp\big(-c\sum_{k=s+1}^t\epsilon_{k}\big)
			&\le C\epsilon_{s}^{-1},
			\quad
			\sum_{t=s+1}^{\infty}\epsilon_{t}
			\exp\big(-c\sum_{k=s+1}^{t-1}\epsilon_{k}\big)
			\le C.
			\label{eq-qd-04}
		\end{align}
	\end{lemma}
	
	\begin{proof}
		The function $Pa-\bar a$ is centered under $\pi$ because invariance gives
		$\int_S\{Pa(x)-\bar a\}\pi(dx)
		=\int_Sa(x)\pi(dx)-\bar a=0$.
		It is also bounded: by \cref{eq-ma-06},
		$\sup_{x\in S}\abs{Pa(x)}\le\sup_{x\in S}\{P(a^2)(x)\}^{1/2}\le C$.
		Apply the Poisson construction above to $Pa-\bar a$. There is a bounded function $\chi$ satisfying
		\begin{equation}
			\chi-P\chi=Pa-\bar a.
			\label{eq-qd-05}
		\end{equation}
		For all sufficiently large $k$, define $V_{k}(x)=1-2\epsilon_{k}\chi(x)$.
		Boundedness of $\chi$ implies
		\begin{equation}
			\frac{1}{2}\le V_{k}(x)\le\frac{3}{2},
			\qquad
			x\in S.
			\label{eq-qd-06}
		\end{equation}
		
		Expanding the integrand gives the one-step Lyapunov inequality
		\begin{align*}
			&P\big\{(1-\epsilon_{k}a)^2V_{k}\big\}(x)
			=1-2\epsilon_{k}\big\{Pa(x)+P\chi(x)\big\}
			+\epsilon_{k}^2P\big(a^2+4a\chi\big)(x)
			-2\epsilon_{k}^3P(a^2\chi)(x).
		\end{align*}
		The final two terms are $O(\epsilon_{k}^2)$ uniformly in $x$, by boundedness of $\chi$ and the uniform conditional second moment of $a$. \Cref{eq-qd-05} gives $Pa(x)+P\chi(x)=\bar a+\chi(x)$.
		Thus
		\begin{equation*}
			P\big\{(1-\epsilon_{k}a)^2V_{k}\big\}(x)
			=1-2\bar a\epsilon_{k}-2\epsilon_{k}\chi(x)+O(\epsilon_{k}^2).
		\end{equation*}
		On the other hand, \cref{eq-qa-08} gives
		\begin{align*}
			\big(1-\bar a\epsilon_{k}\big)V_{k-1}(x)
			&=1-\bar a\epsilon_{k}-2\epsilon_{k-1}\chi(x)
			+2\bar a\epsilon_{k}\epsilon_{k-1}\chi(x)
			=1-\bar a\epsilon_{k}-2\epsilon_{k}\chi(x)+O(\epsilon_{k}^2).
		\end{align*}
		Comparing the preceding two expansions, the difference between the right and left sides is $\bar a\epsilon_{k}+O(\epsilon_{k}^2)$. Therefore, for all sufficiently large $k$,
		\begin{equation}
			P\big\{(1-\epsilon_{k}a)^2V_{k}\big\}(x)
			\le\big(1-c\epsilon_{k}\big)V_{k-1}(x)
			\label{eq-qd-09}
		\end{equation}
		for some $c>0$.
		
		For $k\ge s$, put
		$\mathcal F_{k}^{\mathrm{enl}}=\mathcal F_{k}\vee\sigma(U)$.
		The conditional-independence assumption and the Markov property imply
		that the transition from $X_{k-1}$ to $X_{k}$ remains $P$ in this enlarged
		filtration for every $k>s$.
		For the bridge construction, this follows from
		\cref{lem:exchangeable-construction}.
		Using \cref{eq-qd-01,eq-qd-09},
		\begin{align*}
			\E\big\{H_{k}^2V_{k}(X_{k})\mid\mathcal F_{k-1}^{\mathrm{enl}}\big\}
			&=H_{k-1}^2P\big\{(1-\epsilon_{k}a)^2V_{k}\big\}(X_{k-1})
			\le\big(1-c\epsilon_{k}\big)
			H_{k-1}^2V_{k-1}(X_{k-1}).
		\end{align*}
		Choose a fixed $k_{*}$ so that \cref{eq-qd-06} holds for $k\ge k_{*}$
		and \cref{eq-qd-09} holds for $k>k_{*}$.
		Iteration, $1-x\le e^{-x}$, and \cref{eq-qd-06} prove
		\cref{eq-qd-02} when $s\ge k_{*}$.
		For $s<k\le k_{*}$, the same conditional transition identity gives
		\[
		\E H_{k}^2\le M_{k}\E H_{k-1}^2,\qquad
		M_{k}:=\sup_{x\in S}P\{(1-\epsilon_{k}a)^2\}(x)
		\le 2+2\epsilon_{k}^2\sup_{x\in S}P(a^2)(x)<\infty.
		\]
		Enlarging $C$ by the fixed factor
		$e^{c\sum_{k=1}^{k_{*}}\epsilon_{k}}
		\prod_{k=1}^{k_{*}}\max\{1,M_{k}\}$ covers the early transitions.
		This proves \cref{eq-qd-02} uniformly in $1\le s<t$.
		
		For the left term of \cref{eq-qd-04}, the summation is split at the point $2s$. If $s\le t\le2s$, monotonicity of the stepsizes gives
		\begin{equation*}
			\sum_{k=s+1}^t\epsilon_{k}
			\ge c(t-s)\epsilon_{s}.
		\end{equation*}
		The corresponding geometric sum is bounded by $C\epsilon_{s}^{-1}$. If $t>2s$, an integral comparison yields
		\begin{equation*}
			\sum_{k=s+1}^t\epsilon_{k}
			\ge c\big\{t^{1-\alpha}-s^{1-\alpha}\big\}.
		\end{equation*}
		The tail is therefore bounded by
		\begin{equation*}
			C\int_{2s}^{\infty}
			\exp\big[-c\big\{x^{1-\alpha}-s^{1-\alpha}\big\}\big]dx
			\le Cs^\alpha
			=C\epsilon_{s}^{-1}.
		\end{equation*}
		
		For \cref{eq-qd-04}, for all sufficiently large $t$, $c\epsilon_{t}\le1$ and
		$1-e^{-c\epsilon_{t}}\ge(c/2)\epsilon_{t}$.
		\begin{align*}
			\epsilon_{t}\exp\bigg(-c\sum_{k=s+1}^{t-1}\epsilon_{k}\bigg)
			&\le\frac{2}{c}\exp\bigg(-c\sum_{k=s+1}^{t-1}\epsilon_{k}\bigg)
			\big(1-e^{-c\epsilon_{t}}\big).
		\end{align*}
		The finite initial part is absorbed into $C$, proving \cref{eq-qd-04}.
	\end{proof}
	
	Next, we estimate the recursive perturbation generated by the local substitution at the right end of the block, in order to provide an initial moment bound for applying the above propagation estimation.
    \begin{lemma}\label[lemma]{lem:block-boundary-perturbation}
		For a fixed index $i$, let $\ell_{i}=\min\mathcal B_{i}$,
		$r_{i}=\max\mathcal B_{i}$, and
		$H_{j}^{(i)}=\Delta_{j}-\Delta_{j}^{(i)}$. Choose
		$K_{m}=\lceil C_{1}m^{1/\alpha}\rceil+2m$, where $C_{1}$ is large enough that, whenever $i>K_{m}$, the stepsizes on $\mathcal B_{i}$ are comparable with
		$\epsilon_{i}$ and $m\epsilon_{i}\le1$. Then $H_{j}^{(i)}=0$ for $j<\ell_{i}$, and,
		if $i>K_{m}$,
		\begin{equation}
			\norm{H_{r_{i}}^{(i)}}_{4}
			\le C\sqrt{m}\,\epsilon_{i}.
			\label{eq-qg-01}
		\end{equation}
		If, in addition, $r_{i}<n$ and $s_{i}=r_{i}+1$, then
		\begin{equation}
			\E\abs{H_{s_{i}}^{(i)}}^2
			\le Cm\epsilon_{i}^2.
			\label{eq-qg-02}
		\end{equation}
	\end{lemma}
	
	\begin{proof}
		The two recursions agree before $\ell_{i}$. Summing the difference of their increments from $\ell_{i}$ through $r_{i}$ gives the exact identity
		\begin{align}
			H_{r_{i}}^{(i)}
			={}&-\sum_{j=\ell_{i}}^{r_{i}}\epsilon_{j}
			\big\{\delta(X_{j})-\delta(X_{j}^{(i)})\big\}
			-\sum_{j=\ell_{i}}^{r_{i}}\epsilon_{j}
			\big\{a(X_{j})\Delta_{j-1}
			-a(X_{j}^{(i)})\Delta_{j-1}^{(i)}\big\}.
			\label{eq-qg-03}
		\end{align}
		
		By \cref{eq-pg-01},
		$\delta(X_j)=g_\delta(X_j)-v_\delta(X_{j-1})
		+v_\delta(X_{j-1})-v_\delta(X_j)$.
		The first two terms form a martingale difference. The Burkholder
		inequality and \cref{eq-pg-03} give
		\[
		\Big\|\sum_{j=\ell_i}^{r_i}\epsilon_j
		\big\{g_\delta(X_j)-v_\delta(X_{j-1})\big\}\Big\|_4
		\le C\bigg(\sum_{j=\ell_i}^{r_i}\epsilon_j^2\bigg)^{1/2}.
		\]
		Summation by parts and boundedness of $v_\delta$ bound the remaining
		terms by $C\big(\epsilon_{\ell_i}+\epsilon_{r_i}
		+\sum_{j=\ell_i}^{r_i-1}\abs{\epsilon_{j+1}-\epsilon_j}\big)$.
		Since the stepsizes decrease and are comparable with $\epsilon_i$ on
		$\mathcal B_i$, these estimates give
		\begin{equation}
			\Big\|\sum_{j=\ell_{i}}^{r_{i}}\epsilon_{j}\delta(X_{j})\Big\|_{4}
			\le C\sqrt{m}\,\epsilon_{i}.
			\label{eq-qg-05}
		\end{equation}
		The finite-path distribution identity in
		\cref{lem:exchangeable-construction} gives the same
		bound with $X_{j}^{(i)}$.
		
		For the second line of \cref{eq-qg-03}, conditioning on $\mathcal F_{j-1}$, using the uniform conditional fourth moment of $a$, and applying \cref{lem:delta-moments} give
		\begin{equation*}
			\big\|a(X_{j})\Delta_{j-1}\big\|_{4}
			\le C\epsilon_{j-1}^{1/2}
			\le C\epsilon_{i}^{1/2}.
		\end{equation*}
		The same estimate holds for the updated recursion. Minkowski's inequality now yields
		\begin{equation*}
			\Big\|\sum_{j=\ell_{i}}^{r_{i}}\epsilon_{j}
				a(X_{j})\Delta_{j-1}\Big\|_{4}
			\le Cm\epsilon_{i}^{3/2}.
		\end{equation*}
		Because $m\epsilon_{i}\le1$,
		$m\epsilon_{i}^{3/2}=\sqrt{m}\,\epsilon_{i}\sqrt{m\epsilon_{i}}
		\le\sqrt{m}\,\epsilon_{i}$.
		Combining \cref{eq-qg-03,eq-qg-05} with this inequality proves \cref{eq-qg-01}.
		
		If $s_{i}=r_{i}+1$, both paths use the same state $X_{s_{i}}$. Their difference therefore evolves for this one step as
		\begin{equation}
			H_{s_{i}}^{(i)}
			=\big(1-\epsilon_{s_{i}}a(X_{s_{i}})\big)H_{r_{i}}^{(i)}.
			\label{eq-qg-07}
		\end{equation}
		The two factors in \cref{eq-qg-07} need not be independent. H\"older's inequality, the fourth moment of $a(X_{s_{i}})$, and \cref{eq-qg-01} give
		\begin{equation*}
			\norm{H_{s_{i}}^{(i)}}_{2}
			\le\norm{1-\epsilon_{s_{i}}a(X_{s_{i}})}_{4}
			\norm{H_{r_{i}}^{(i)}}_{4}
			\le C\sqrt{m}\,\epsilon_{i}.
		\end{equation*}
		Squaring this bound proves \cref{eq-qg-02}.
	\end{proof}

    By using Poisson decomposition and recursive moment bounds, the following estimates the bias-corrected weighted sum to control $\E|U_c|$.
	\begin{lemma}\label[lemma]{lem:weighted-adaptive-sum}
		Under \cref{ass:density-mixing,ass:fourth-moments}, we follow that
		\begin{equation}
		\begin{aligned}
		\E\abs{
		\frac{q_{0}\Delta_{0}}{\epsilon_{0}}
		-\sum_{i=1}^{n}q_{i}h(X_{i})\Delta_{i-1}
		-\kappa A_{n}-\mu_{\nu,n}}
		\le Cn^{(1-\alpha)/2}.
		\end{aligned}
		\end{equation}
	\end{lemma}

	\begin{proof}
		From \cref{eq-qa-01a}, summation by parts gives
		\begin{align*}
		\sum_{i=1}^{n}q_{i}h(X_{i})\Delta_{i-1}
		={}&\sum_{i=1}^{n}q_{i}\Delta_{i-1}\xi_{i}
		+q_{1}\Delta_{0}v_{h}(X_{0})
		-q_{n}\Delta_{n-1}v_{h}(X_{n})\\
		&+\sum_{i=1}^{n-1}(q_{i+1}-q_{i})\Delta_{i}v_{h}(X_{i})
		-\sum_{i=1}^{n-1}q_{i}\epsilon_{i}
		\big\{a(X_{i})\Delta_{i-1}+\delta(X_{i})\big\}v_{h}(X_{i}).
		\end{align*}
		After subtracting $\kappa A_n$, the last term is
		\[
		\sum_{i=1}^{n-1}q_i\epsilon_i
		\big\{\delta(X_i)v_h(X_i)-\kappa\big\}.
		\]
		Conditional orthogonality, \cref{eq-pg-03,lem:delta-moments}, and
		\cref{eq-ph-01} imply
		\[
		\E\abs{\sum_{i=1}^{n}q_i\Delta_{i-1}\xi_i}
		\le C\bigg(1+\sum_{i=2}^{n}\epsilon_{i-1}\bigg)^{1/2}
		\le Cn^{(1-\alpha)/2}.
		\]
		The boundary terms are bounded in $L^1$ by $C$. Furthermore,
		\cref{eq-ph-01,lem:delta-moments,eq-ma-06} give
		\[
		\sum_{i=1}^{n-1}
		\bigg\{\abs{q_{i+1}-q_i}\E\abs{\Delta_i}
		+\abs{q_i}\epsilon_i\E\abs{a(X_i)\Delta_{i-1}}\bigg\}
		\le C\sum_{i=1}^{n}\epsilon_i^{3/2}
		\le Cn^{(1-\alpha)/2}.
		\]
		Finally, $\delta v_h-\kappa$ is centered under $\pi$ by \cref{eq-mb-06},
		and \cref{eq-pe-02,eq-ph-01} yield
		\[
		\E\bigg[
		\sum_{i=1}^{n-1}q_i\epsilon_i
		\big\{\delta(X_i)v_h(X_i)-\kappa\big\}
		\bigg]^2\le C.
		\]
		The remaining deterministic terms are bounded by
		\cref{eq-ph-01} and $\abs{\mu_{\nu,n}}\le C$ from
		\cref{lem:oracle-variance}. Combining these estimates proves the claim.
	\end{proof}

    Finally, by combining the estimation of the block boundary perturbation and the subsequent propagation, the weighted sum change caused by the local path replacement is controlled, and an upper bound for $\E|U_c-U_c'|^2$ is established.
    \begin{lemma}\label[lemma]{lem:bridge-replacement-sum}
		Let $K_m$ be as in \cref{lem:block-boundary-perturbation}. For every
		$1\le i\le n$,
		\begin{equation}
		\E\bigg[
		\sum_{j=1}^{n}q_j
		\big\{-h(X_j)\Delta_{j-1}
		+h(X_j^{(i)})\Delta_{j-1}^{(i)}\big\}
		\bigg]^2
		\le
		\begin{cases}
		Cm^2, & i\le K_m,\\
		C\big\{m\epsilon_i+m^2\epsilon_i^2\big\}, & i>K_m.
		\end{cases}
		\end{equation}
	\end{lemma}

	\begin{proof}
		Fix $i$. For $i>K_m$, split the sum into its local part $L_i$,
		containing the indices $\ell_i,\ldots,r_i+1$ when $r_i<n$
		and $\ell_i,\ldots,n$ otherwise, and its remaining part $L_i'$.
		By \cref{eq-qa-01a}, summation by parts over the local range gives,
		when $r_i<n$,
		\begin{align*}
		\sum_{j=\ell_i}^{r_i+1}q_jh(X_j)\Delta_{j-1}
		={}&\sum_{j=\ell_i}^{r_i+1}q_j\Delta_{j-1}\xi_j
		+q_{\ell_i}\Delta_{\ell_i-1}v_h(X_{\ell_i-1})
		-q_{r_i+1}\Delta_{r_i}v_h(X_{r_i+1})\\
		&+\sum_{j=\ell_i}^{r_i}(q_{j+1}-q_j)\Delta_jv_h(X_j)
		-\sum_{j=\ell_i}^{r_i}q_j\epsilon_j
		\big\{a(X_j)\Delta_{j-1}+\delta(X_j)\big\}v_h(X_j).
		\end{align*}
		The same identity with upper endpoint $n$ applies when $r_i=n$.
		Using \cref{eq-ph-01,eq-pg-03,lem:delta-moments,eq-pe-02} and
		$m\epsilon_i\le1$ gives
		\[
		\E L_i^2\le C\big(m\epsilon_i+m^2\epsilon_i^2\big).
		\]
		The finite-path distribution identity in
		\cref{lem:exchangeable-construction} gives the same estimate for the
		updated path and hence for $L_i$.

		If $r_i\le n-2$, then for $j\ge r_i+2$ the paths use the same states and
		\[
		H_j^{(i)}=\big(1-\epsilon_j a(X_j)\big)H_{j-1}^{(i)}.
		\]
		By \cref{lem:homogeneous-propagation,eq-qg-02} and the comparability
		of the stepsizes on $\mathcal B_i$,
		\[
		\sum_{j=r_i+2}^{n}\E\abs{H_{j-1}^{(i)}}^2
		\le C\epsilon_{r_i+1}^{-1}\E\abs{H_{r_i+1}^{(i)}}^2
		\le Cm\epsilon_i.
		\]
		Here
		\[
		L_i'=-\sum_{j=r_i+2}^{n}q_jh(X_j)H_{j-1}^{(i)}.
		\]
		The transition identity in \cref{lem:exchangeable-construction} keeps
		\cref{eq-qa-01a} valid after conditioning on the bridge seed.
		Summation by parts gives
		\begin{align*}
		-L_i'={}&\sum_{j=r_i+2}^{n}q_jH_{j-1}^{(i)}\xi_j
		+q_{r_i+2}H_{r_i+1}^{(i)}v_h(X_{r_i+1})
		-q_nH_{n-1}^{(i)}v_h(X_n)\\
		&+\sum_{j=r_i+2}^{n-1}
		\big\{(q_{j+1}-q_j)H_j^{(i)}
		-q_j\epsilon_ja(X_j)H_{j-1}^{(i)}\big\}v_h(X_j).
		\end{align*}
		Together with \cref{eq-ph-01}, and
		\cref{lem:homogeneous-propagation,eq-qd-04} therefore give
		\[
		\E\abs{L_i'}^2
		\le C\sum_{j=r_i+2}^{n}\E\abs{H_{j-1}^{(i)}}^2
		+C\E\abs{H_{r_i+1}^{(i)}}^2
		\le Cm\epsilon_i.
		\]
		If $r_i\ge n-1$, then $L_i'=0$ and the same estimate holds.
		Combining the local and tail bounds proves the second line.

		For $i\le K_m$, each local summand has uniformly bounded second moment:
		\[
		\E\big[ h(X_j)^2\Delta_{j-1}^2\big]
		=\E\big[\Delta_{j-1}^2P(h^2)(X_{j-1})\big]\le C,
		\]
		and the finite-path identity gives the same bound for the updated path.
		There are $O(m)$ local indices, so $\E L_i^2\le Cm^2$. If a tail is present,
		\[
		\E\abs{H_{r_i+1}^{(i)}}^2
		\le2\E\Delta_{r_i+1}^2+
		2\E\big(\Delta_{r_i+1}^{(i)}\big)^2
		\le C\epsilon_{r_i+1}.
		\]
		The propagation and summation-by-parts estimates above give
		$\E\abs{L_i'}^2\le C$. Thus the first line follows.
	\end{proof}
	
	Based on the aforementioned auxiliary lemma, the proof of \cref{lem:corrected-remainder} is completed below.
	\begin{proof}[Proof of \cref{lem:corrected-remainder}]
		From \cref{eq-pa-02,lem:weighted-adaptive-sum} and using the positive lower bound of $\sigma_{\nu,n}$ guaranteed by \cref{lem:oracle-variance}, we have
		\[
		\E|U_c|
		\le
		\frac{Cn^{(1-\alpha)/2}}{\sigma_{\nu,n}\sqrt n}
		\le Cn^{-\alpha/2},
		\]
		Thus, we obtain \cref{eq-qc-02}.
		
		For the second estimate, from \cref{eq-pa-02}, the deterministic terms in $U_c - U_c'$ cancel out. Taking the conditional expectation of the uniform index $I$ and applying
		\cref{lem:bridge-replacement-sum}, we obtain
		\[
		\E|U_c-U_c'|^2
		\le \frac{C}{n^2\sigma_{\nu,n}^2}
		\left\{m^{2+1/\alpha}+mn^{1-\alpha}\right\}
		=
		C\sigma_{\nu,n}^{-2}
		\left\{
		\frac{m^{2+1/\alpha}}{n^2}
		+\frac{m}{n^{1+\alpha}}
		\right\}.
		\]
		This proves \cref{eq-qh-03}.
	\end{proof}
	
\end{document}